\documentclass[11pt,reqno]{article}
\usepackage{amsmath,amsfonts,amsthm,color,amssymb, mathrsfs}

 \usepackage{pxfonts}
\usepackage{comment}
\makeatletter
\@tfor\@tempa:=
\alpha \beta \gamma \delta \epsilon \zeta \eta \theta \iota \kappa \lambda \mu
\nu \xi \pi \rho \sigma \tau \upsilon \phi \chi \psi \omega \varepsilon
\vartheta \varpi \varrho \varsigma \varphi  
\do {
      \count255 \numexpr\@tempa+"7000\relax
      \expandafter\mathchardef\@tempa \count255 }
\makeatother

\usepackage{xcolor}
\usepackage[T1]{fontenc}
\usepackage{wasysym}
\usepackage[normalem]{ulem}
\usepackage{stmaryrd} 

\usepackage[left=1in, right=1in, top=1.1in,bottom=1.1in]{geometry}
\usepackage{enumitem}
\usepackage{hyperref}
\hypersetup{
	colorlinks   = true,
	citecolor    = blue,
	linkcolor=blue
}
\allowdisplaybreaks
\usepackage{csquotes}
\usepackage{graphicx}
\usepackage{pstricks}
\usepackage{lmodern}

\newtheorem{lemma}{Lemma}
\newtheorem{theorem}{Theorem}[section]

\newtheorem{proposition}{Proposition}

\def\rme{{\rme}}

\def\blue{\color{blue}}
\def\red{\color{red}}

\def\la{\left\langle}
\def\ra{\right\rangle}

\def\R{\mathbb R}
\def\E{\mathbb E}

\def\blue{\color{blue}}
\def\red{\color{red}}
\def\bR{\mathbb R}

\def\cF{\mathcal F}
\def\cH{\mathcal H}
\def\cF{\mathcal F}

\def\cX{\mathcal X}

\def\bbt{{\boldsymbol{t}}}

\def\bbs{{\boldsymbol{s}}}

\def\la{\langle}
\def\ra{\rangle}

\newcommand{\bxi}{\boldsymbol\xi}
\newcounter{bean}
\newcommand{\benuma}{\setlength{\labelwidth}{.25in}
	
	\begin{list}
		{(\alph{bean})}{\usecounter{bean}}}
	\newcommand{\eenuma}{\end{list}}

\begin{document}
\title{Central limit theorems for stochastic heat equation driven by Gaussian noise with heat-kernel covariance}
\author{Meng Wang\thanks{Center for Applied Mathematics, Tianjin University, Tianjin, China. wmeng\_2206@tju.edu.cn}\and 
Wangjun Yuan\thanks{Department of Mathematics, Southern University of Science and Technology, Shenzhen, China. ywangjun@connect.hku.hk.}}
\date{}
\maketitle
\begin{abstract}
In this article, we study the spatial fluctuations of the Skorohod solution to a
stochastic heat equation on $\R^d$ for $d<4$, driven by multiplicative Gaussian noise
with the non-separable covariance kernel $p_{|t-s|}(x-y)$. We prove that the
solution is strictly stationary and spatially ergodic at every fixed time.
For the centered spatial integral
$F_R(t)=\int_{\{|x|<R\}}(u(t,x)-1)dx,$
we show that
$\mathbb E[F_R(t)F_R(s)]\sim K(t,s)R^d$ as $R\to\infty$. Using moment
estimates for the first two Malliavin derivatives and a second-order Gaussian
Poincar\'e inequality, we establish a quantitative central limit theorem in
total variation distance with rate $R^{-d/2}$. We also prove a functional
central limit theorem for the process $\{R^{-d/2}F_R(t)\}_{t\geq0}$.

\vspace{0.2cm}
\end{abstract}

\section{Introduction}

For $d<4$, we consider the following stochastic heat equation
\begin{equation}\label{eq-SPDE}
\left\{
\begin{aligned}
    \frac{\partial}{\partial t}u(t,x)
    &=\frac 12
    \Delta u(t,x)+u(t,x)\dot W(t,x),
    \qquad t\geq0,\ x\in\R^d,\\
    u(0,x)&=1,
\end{aligned}
\right.
\end{equation}
where $\dot W$ is a generalized space--time Gaussian field with covariance
kernel
\begin{align}\label{eq-covariance}
    \E\big[\dot W(t,x)\dot W(s,y)\big]
    =
    p_{|t-s|}(x-y),
\end{align}
where $p_t(x)= (2\pi t)^{-\frac{d}{2}} \exp (-\frac{|x|^2}{2t})$  for $t>0$ and $x\in\R^d$ is the heat kernel, and the product in $u\dot W$ is a Wick product. Equation \eqref{eq-SPDE} corresponds to the case $\alpha=2$ in \cite{23swy}, where the existence and uniqueness of the square-integrable Skorohod solution were established for $d<4$.

For $R>0$, let $B_R=\{x\in\R^d:|x|\le R\}$ and denote
\begin{equation}\label{eq:spatial-average}
    F_R(t)=\int_{B_R}\bigl(u(t,x)-1\bigr)dx.
\end{equation}

In this paper, we study the spatial ergodicity of the solution to
\eqref{eq-SPDE} and the asymptotic behavior of the centered spatial
integrals $F_R(t)$ as $R\to\infty$.
A distinctive feature of the noise in \eqref{eq-SPDE} is that its covariance
is not separable in time and space. More precisely, the heat kernel
$p_{|t-s|}(x-y)$ cannot be written in the product form
$\gamma_0(t-s)\gamma_1(x-y)$, which is commonly assumed in the literature on
temporally correlated Gaussian noise.

We first recall some related results for stochastic heat equations driven by
Gaussian noises with separable covariance structures. For the one-dimensional
equation driven by space--time white noise, Huang,
Nualart, and Viitasaari \cite{HNV20} proved that
$\operatorname{Var}(F_R(t))$ is of order $R$ and established a quantitative
central limit theorem in total variation distance with rate $R^{-1/2}$,
together with a functional central limit theorem. These results were extended in \cite{HNVZ20} to the $d$-dimensional
stochastic heat equation driven by Gaussian noise with covariance
$\delta_0(t-s)|x-y|^{-\beta}$, $0<\beta<d\wedge2$.  In this case,
$\operatorname{Var}(F_R(t))$ is of order $R^{2d-\beta}$ and the rate in total
variation distance is $R^{-\beta/2}$. More generally, Chen, Khoshnevisan, Nualart, and Pu
\cite{CKNP2021} studied stochastic heat equations driven by Gaussian noise with
covariance
$\delta_0(t-s)f(x-y),$ 
where $f$ is a spatially homogeneous covariance satisfying Dalang's condition.
They established the spatial stationarity and ergodicity of the solution under
suitable conditions on $f$ by means of Poincar\'e-type inequalities.
Subsequently, in \cite{CKNP2023}, they obtained quantitative and
functional central limit theorems for spatial averages of functions of the
solution by using the Malliavin--Stein method; also see \cite{CDNP22ParabolicCLT,CDNP22Delta}.

For temporally correlated Gaussian noises, most existing works considered
separable covariance structures of the form
$\E\big[\dot W(t,x)\dot W(s,y)\big]
=
\gamma_0(t-s)\gamma_1(x-y).$
Under this framework, Nualart and Zheng
\cite{Nualart2020Averaging} established qualitative and functional central
limit theorems for integrable and Riesz-type spatial correlations, while
Nualart, Song, and Zheng \cite{NSZ21} treated the one-dimensional rough
spatial case. These results were subsequently strengthened by Nualart, Xia,
and Zheng \cite{NXZ22}, who obtained quantitative central limit theorems in
these settings. Related
quantitative and functional central limit theorems for the parabolic Anderson
model driven by time-independent Gaussian noise were obtained by Balan and
Yuan \cite{BY23}. Related quantitative and functional central limit theorems for spatial
averages of stochastic wave equations can be found in
\cite{NualartZheng20WaveErgodicity,NualartZheng22WaveCLT,bnz22,
BalanYuan22Wave,BalanYuan24WaveRough,BHWXY23,24bz,BalanZheng26}.


Besides the above developments on spatial fluctuations of SPDEs, another
motivation for studying \eqref{eq-SPDE} comes from directed polymer models in
correlated random environments. Shen, Song, Sun, and Xu
\cite{ShenSongSunXu21} considered a directed polymer in a space--time
correlated environment generated by a Poisson field of independent random
walks and proved that, under a suitable scaling, the corresponding partition
function converges weakly to the Stratonovich solution of \eqref{eq-SPDE}.
Song, Wang, and Yuan \cite{23swy} established the existence and uniqueness of
the corresponding Stratonovich and Skorohod solutions and derived their
Feynman--Kac representations. Motivated by these works, we investigate the
spatial Gaussian fluctuations of the Skorohod solution of \eqref{eq-SPDE}. 


The preceding results for Gaussian noises with separable covariances
do not apply directly to the heat-kernel covariance considered here.
To handle this covariance, we combine the Wiener chaos expansion
with Schur's test to establish moment estimates for the first
and second Malliavin derivatives of the solution for all $d<4$.
We now state the main result. 


Let $\omega_d$ be the
Lebesgue measure of the unit ball in $\R^d$.  Set
\begin{equation}
\sigma_R^2(t)=\text{Var}(F_R(t)),\,\,\,\,\rho_{t,s}(x-y)=\E\big[(u(t,x)-1)(u(s,y)-1)\big].
\end{equation}
As we show below, the covariance $\rho_{t,s}(x-y)$ depends only
on the difference $x-y$.
\begin{theorem}\label{thm:main}Let $d<4$. Then the following statements hold.
\begin{enumerate}[label=\textup{(\roman*)}]
\item For every $t>0$, the random field
$\{u(t,x):x\in\R^d\}$ is strictly stationary and ergodic.
\item For any $t>0$ and $s>0$, 
\begin{align*}
\E\big[F_R(t)F_R(s)\big]\sim K(t,s)R^d\, \text{as}\,\,R\to\infty,\,\,\text{where} \,\,K(t,s):=\omega_d\int_{\R^d}\rho_{t,s}(z)dz.
\end{align*}
\item For any $t>0$,
\begin{align*}
d_{TV}\left(\frac{F_R(t)}{\sigma_R(t)},Z\right)\leq C_tR^{-\frac{d}{2}},
\end{align*}
where $Z\sim \mathcal N(0,1)$ and $C_t$  is a finite constant
independent of $R$.
\item For every $R>0$, the process
$\{R^{-d/2}F_R(t)\}_{t\ge0}$ has a continuous modification.  Moreover, as
$R\to\infty$,
\begin{align*}
\{R^{-d/2}F_R(t)\}_{t\geq0}
\overset{d}{\longrightarrow}
\{Q(t)\}_{t\geq0}
\end{align*}
in $C[0,\infty)$, where $\{Q(t)\}_{t\geq0}$ is a centered Gaussian process
with covariance
$\E\big[Q(t)Q(s)\big]=K(t,s)$.
\end{enumerate}
\end{theorem}


The rest of the paper is organized as follows. In Section~\ref{sec:preliminaries}, we recall some preliminaries on Malliavin calculus. Section~\ref{sec:estimates} is devoted to moment estimates for the first and second Malliavin derivatives of the solution. In Section~\ref{sec:main-results}, we prove the four assertions of Theorem~\ref{thm:main}.

\section{Preliminaries}\label{sec:preliminaries}
In this section, we provide some preliminaries on the Malliavin calculus. We refer to \cite{Nualart,Hu} for more details. 

Let $\cH$ be the completion of the
space $C_c^\infty(\R_+\times\R^d)$ of smooth functions with compact
support, equipped with the inner product
\begin{align} \label{eq-H norm}
    \langle f,g \rangle_{\cH}
    =& \int_{\bR_+^{2}} \int_{\bR^{2d}} f(s,x) g(t,y) p_{|t-s|}(x-y) dxdy dsdt \nonumber \\
    =& c_d\int_{\bR_+^{2} } \int_{\bR^d} \cF f(s,\xi)\overline{\cF g(t,\xi)} \cF p_{|t-s|}(\xi) d\xi dsdt.
\end{align}
Here $\cF$ denotes the Fourier transform in the spatial variables only;
in particular, $\cF f(s,\xi)=\int_{\R^d}e^{-i\xi\cdot x}f(s,x)dx$.
With this convention, $c_d=(2\pi)^{-d}$ and
$\cF p_t(\xi)=e^{-t|\xi|^2/2}$ for $t>0$.
We denote by $\|\cdot\|_{\cH}$ the norm induced by the inner product.  

Let $W=\{W(f), f\in\mathcal H\}$ be  an isonormal Gaussian process defined on a
complete probability space $(\Omega, \mathcal F, P)$, with covariance 
\[\E[W(f)W(g)]=\la f, g \ra_\mathcal H, \,\text{ for } f,g\in \mathcal H. \]
For $f\in \mathcal H^{\otimes m}$, denote the  symmetrization of $f$ by 
\[\tilde  f(t_1, x_1, \dots, t_m,x_m) := \frac1{m!}\sum_{\sigma\in S_m} f(t_{\sigma(1)}, x_{\sigma(1)}, \dots,t_{\sigma(m)}, x_{\sigma(m)}),\]
where $S_m$ is the set of all permutations of $\{1,2,\dots,m\}.$  Let ${\mathcal H}^{\tilde \otimes m}$ be the symmetrization  of ${\mathcal H}^{\otimes m}$. Then for $f\in {\mathcal H}^{\tilde \otimes m}$,  one can define the $m$th multiple Wiener integral $I_m(f)$ by a limiting argument.  Moreover, for $f \in {\mathcal H}^{\tilde \otimes m}$ and  $g\in {\mathcal  H}^{\tilde \otimes n}$, we have
\begin{equation*}\label{e:ImIn}
	{\mathbb E}[I_m(f)I_n(g)]=m!\langle f,g\rangle_{\mathcal  H^{\otimes m}}\delta_{mn}, 
\end{equation*}
where $\delta_{mn}$ is the Kronecker delta function.  For $f\in {\mathcal  H}^{\otimes m}$ which is not necessarily symmetric, we simply define $I_m(f):= I_m(\tilde f).$  For a square integrable random variable $F$ which is measurable with respect to the $\sigma$-algebra generated by $W$, it has the following unique expansion (Wiener chaos expansion):
\begin{equation}\label{weiner chaos}
F=\E[F]+\sum_{n=1}^\infty I_n(f_n), ~ f_n\in \mathcal H^{\tilde  \otimes n}.
\end{equation}

  Let $F$ be a smooth  cylindrical random variable, i.e., $F$ is of the form
$	F=f(W(\phi_1),\cdots,W(\phi_n)),$
where $\phi_i\in \mathcal H$ and $f$ is a smooth function of which all the derivatives are of polynomial growth. Then the Malliavin derivative $DF$ of $F$ is defined by
\begin{align*}
	DF=\sum\limits_{j=1}^n\frac{\partial f}{\partial x_j}(W(\phi_1),\cdots,W(\phi_n))\phi_j.
\end{align*}
Noting that the operator $D$ is closable from $L^2(\Omega)$ into $L^2(\Omega;\mathcal H)$, we can define the Sobolev space $\mathbb D^{1,2}$ as the closure of the space $\mathfrak S$ of all  smooth cylindrical random variables under the norm
\begin{align*}
	\|F\|_{1,2}=\left(\E[F^2]+\E[\|DF\|^2_{\mathcal H}]\right)^{1/2}.
\end{align*}
Similarly, one can define the $k$th Malliavin derivative $D^k F$ as an $\mathcal H^{\otimes k}$-valued variable for $k\ge 2$, and for any $p>1$ let $\mathbb D^{k,p}$ be the completion 
of $\mathfrak S$ under the norm \[\|F\|_{k,p}=\bigg(\E[|F|^p]+\sum_{j=1}^k\E[\|D^jF\|^p_{\mathcal H^{\otimes j}}]\bigg)^{1/p}.\]

The divergence operator $\delta$ (also called Skorohod integral) is defined by the duality 
\begin{align*}
	\E[\delta(u)F]=\E[\la DF,u\ra_{\mathcal H}]
\end{align*}
for $F\in\mathbb D^{1,2}$ and $u\in \text{Dom}(\delta)$, {where $\operatorname{Dom}(\delta)$ consists of all
$u\in L^2(\Omega;\mathcal H)$ such that
\[
\big|\E[\langle DF,u\rangle_{\mathcal H}]\big|
\leq c_u\|F\|_2.
\]} 

We next recall the Ornstein-Uhlenbeck operator and its pseudo-inverse, which will be used in the Malliavin--Stein analysis below; see, e.g., \cite[Section~1.4]{Nualart}.

Suppose that $F\in L^2(\Omega)$ admits the Wiener chaos expansion 
\eqref{weiner chaos}.
The Ornstein--Uhlenbeck generator $L$ is defined by
\begin{align*}
LF
=
-\sum_{n\geq1}nI_n(f_n),
\end{align*}
provided that the series converges in $L^2(\Omega)$. It is well known
that $F\in\operatorname{Dom}(L)$ if and only if
$F\in\mathbb D^{1,2}$ and $DF\in\operatorname{Dom}(\delta)$. In this
case,
\begin{align*}
LF=-\delta(DF).
\end{align*}

The pseudo-inverse $L^{-1}$ of $L$ is defined by
\begin{align*}
L^{-1}F
=
-\sum_{n\geq1}\frac{1}{n}I_n(f_n).
\end{align*}
In particular,
\begin{align*}
LL^{-1}F
=
F-\E[F].
\end{align*}

For every $F\in\mathbb D^{1,2}$ satisfying $\E[F]=0$, the random
element
$-DL^{-1}F$
belongs to $\operatorname{Dom}(\delta)$, and
\begin{align}\label{e:F-delta-DL}
F
=
\delta\left(-DL^{-1}F\right).
\end{align}

\section{Estimates on the Malliavin derivatives}\label{sec:estimates}
By \cite{23swy}, for $d<4$, equation \eqref{eq-SPDE} admits a unique
square-integrable mild Skorohod solution, which we denote by $u$.
 In this section, we establish moment estimates for the first and second
Malliavin derivatives of the solution. 

To this end, we first introduce an auxiliary space $\cX$, defined as
the completion of $C_c^\infty(\R_+\times\R^d)$ equipped with the
inner product
\begin{equation*}
\langle f,g\rangle_{\cX}
=c_d\int_0^\infty\int_{\R^d}
\cF f(r,\xi)\overline{\cF g(r,\xi)}\,\mu(d\xi)\,dr,
\end{equation*}
where $\mu(d\xi)=\frac{d\xi}{1+|\xi|^2}$.


The following proposition allows us to control the $\cH^{\otimes n}$-norms
by the corresponding $\cX^{\otimes n}$-norms, which will be used to
establish the Malliavin derivative estimates.
\begin{proposition}\label{lem:full-schur}
For every $T>0$, there exists $C_T<\infty$ such that, for every
$f\in\cX$ supported in $[0,T]\times\R^d$, we have $f\in\cH$ and
\begin{equation}\label{eq:full-H-X}
\|f\|_{\cH}^2\le C_T\|f\|_{\cX}^2.
\end{equation}
Moreover, for every $n\ge1$ and every $f\in\cX^{\otimes n}$ supported
in $[0,T]^n\times\R^{dn}$, we have $f\in\cH^{\otimes n}$ and
\begin{equation}\label{eq:full-tensor-embedding}
\|f\|_{\cH^{\otimes n}}^2
\le C_T^n\|f\|_{\cX^{\otimes n}}^2.
\end{equation}
\end{proposition}
\begin{proof}
It suffices to prove the estimates on the corresponding smooth
compactly supported subspaces; the general case follows by density
and continuity.

Fix $\xi\in\R^d$.
If $\xi\ne0$, then
\begin{align*}
&\int_0^T e^{-\frac12|r-s||\xi|^2}ds
=\int_0^r e^{-\frac12(r-s)|\xi|^2}ds
+\int_r^T e^{-\frac12(s-r)|\xi|^2}ds\\
&=\frac{2}{|\xi|^2}
\left(2-e^{-r|\xi|^2/2}
-e^{-(T-r)|\xi|^2/2}\right)\le\frac4{|\xi|^2}.
\end{align*}
Hence,
\begin{align}\label{eq:Scur}
\sup_{r\in[0,T]}
\int_0^T e^{-\frac12|r-s||\xi|^2}ds
\le \min\left\{T,\frac4{|\xi|^2}\right\}
\le\frac{C_T}{1+|\xi|^2},
\end{align}
where  one may
take $C_T=\max\{2T,8\}$. 
For $\xi=0$, the time integral equals $T$, so the same bound holds.

{For fixed $\xi\in\R^d$, we apply
Lemma~\ref{lem:full-schur-test} with 
\[
K(r,s)=e^{-\frac12|r-s||\xi|^2},
\qquad
g(r)=\cF f(r,\xi),
\qquad
A=\frac{C_T}{1+|\xi|^2},
\]
then we have
\[
\left|\int_0^T\int_0^T
e^{-\frac12|r-s||\xi|^2}
\cF f(r,\xi)\overline{\cF f(s,\xi)}
drds\right|
\le \frac{C_T}{1+|\xi|^2}
\int_0^T|\cF f(r,\xi)|^2dr.
\]
Multiplying by $c_d$ and integrating in $\xi$ gives \eqref{eq:full-H-X}.
}

For the tensor-product estimate \eqref{eq:full-tensor-embedding}, write
\[
\boldsymbol r=(r_1,\ldots,r_n),\qquad
\boldsymbol s=(s_1,\ldots,s_n),\qquad
\bxi=(\xi_1,\ldots,\xi_n),
\]
with
$d\boldsymbol r=dr_1\cdots dr_n,
d\boldsymbol s=ds_1\cdots ds_n.$
{For fixed $\bxi\in\R^{dn}$, we have by \eqref{eq:Scur}
\begin{align*}
&\sup_{\boldsymbol r\in[0,T]^n}
\int_{[0,T]^n}
\prod_{j=1}^n
e^{-\frac12|r_j-s_j||\xi_j|^2}
d\boldsymbol s=\prod_{j=1}^n
\sup_{r_j\in[0,T]}
\int_0^T e^{-\frac12|r_j-s_j||\xi_j|^2}ds_j\le C_T^n\prod_{j=1}^n\frac1{1+|\xi_j|^2}.
\end{align*}
Since $f\in\cX^{\otimes n}$, the function
$\boldsymbol r\mapsto\cF f(\boldsymbol r,\bxi)$ belongs to
$L^2([0,T]^n)$ for almost every $\bxi$.
For each such $\bxi$,  applying
Lemma~\ref{lem:full-schur-test} yields
\begin{align*}
&\left|
\int_{[0,T]^n}\int_{[0,T]^n}
\prod_{j=1}^n
e^{-\frac12|r_j-s_j||\xi_j|^2}
\cF f(\boldsymbol r,\bxi)
\overline{\cF f(\boldsymbol s,\bxi)}
d\boldsymbol sd\boldsymbol r
\right|\le
C_T^n\prod_{j=1}^n\frac1{1+|\xi_j|^2}
\int_{[0,T]^n}|\cF f(\boldsymbol r,\bxi)|^2
d\boldsymbol r.
\end{align*}
Multiplying by $c_d^n$ and integrating in $\bxi$, we obtain
\begin{align*}
\|f\|_{\cH^{\otimes n}}^2
&\le c_d^nC_T^n
\int_{\R^{dn}}\int_{[0,T]^n}
|\cF f(\boldsymbol r,\bxi)|^2
d\boldsymbol r
\prod_{j=1}^n\frac{d\xi_j}{1+|\xi_j|^2}=C_T^n\|f\|_{\cX^{\otimes n}}^2,
\end{align*}
which proves \eqref{eq:full-tensor-embedding}.}
\end{proof}
In the above proof, we have used the following special case of Schur's test.
 \begin{lemma}\label{lem:full-schur-test}
Let $T>0$ and let $K:[0,T]^2\to\R$ be a symmetric
measurable kernel satisfying
\[
\sup_{r\in[0,T]}\int_0^T|K(r,s)|ds\le A<\infty.
\]
Then, for every $g\in L^2(0,T)$,
\[
\left|\int_0^T\int_0^T
K(r,s)g(r)\overline{g(s)}dsdr\right|
\le A\int_0^T|g(r)|^2dr.
\]
The same conclusion holds on $[0,T]^n$, with $r,s$
replaced by $\boldsymbol r,\boldsymbol s$ and integration
taken with respect to $n$-dimensional Lebesgue measure.
\end{lemma}

\begin{proof}
By $2|g(r)||g(s)|\le |g(r)|^2+|g(s)|^2$ and 
the symmetry of $K$, we have
\begin{align*}
&\int_0^T\int_0^T
|K(r,s)||g(r)||g(s)|dsdr\\
&\le
\frac12\int_0^T\int_0^T
|K(r,s)|\Big(|g(r)|^2+|g(s)|^2\Big)dsdr\\
&=
\int_0^T|g(r)|^2
\left(\int_0^T|K(r,s)|ds\right)dr\\
&\le A\int_0^T|g(r)|^2dr.
\end{align*}
This proves absolute integrability and the desired inequality.
The proof on $[0,T]^n$ is identical.
\end{proof}

For $N>1$, set
\begin{equation*}
D_N=\int_{|\xi|\le N}\mu(d\xi),
\qquad
C_N=\int_{|\xi|>N}\frac{\mu(d\xi)}{|\xi|^2}.
\end{equation*}
The following lemma records the integrability properties of $\mu$
needed for the subsequent kernel estimates.
\begin{lemma}\label{lem:full-DN-CN}
	If $d<4$, then $D_N<\infty$ for every $N>1$, and
	$C_N\to0$ as $N\to\infty$.
\end{lemma}

\begin{proof}
	Since $(1+|\xi|^2)^{-1}\le1$,
	\begin{align*}
		D_N
		&=\int_{|\xi|\le N}\frac{d\xi}{1+|\xi|^2}
		\le \big|\{|\xi|\le N\}\big|<\infty.
	\end{align*}
    Using the change of variables $\xi=N\eta$ and polar coordinates,
we obtain
\begin{align*}
0\le C_N
&=\int_{|\xi|>N}\frac{d\xi}{|\xi|^2(1+|\xi|^2)}
\le \int_{|\xi|>N}\frac{d\xi}{|\xi|^4}\\
&=N^{d-4}\int_{|\eta|>1}\frac{d\eta}{|\eta|^4}\\
&=d\omega_d N^{d-4}\int_1^\infty r^{d-5}\,dr\\
&=\frac{d\omega_d}{4-d}N^{d-4}\longrightarrow0,
\end{align*}
as $N\to\infty$ where we used $d<4$.
\end{proof}
\subsection{Estimate for the first-order Malliavin derivative}
\label{subsec:full-first-malliavin-derivative}
Recall that the solution $u$ to equation \eqref{eq-SPDE} has the following chaos expansion
\begin{equation}\label{eq:full-solution-chaos}
u(t,x)=1+\sum_{n\ge1}I_n(f_n(\cdot;t,x)),
\end{equation}
where for  $n\ge1$, 
\begin{align*}f_n(t_1,x_1,\ldots,t_n,x_n;t,x)=\mathbf 1_{\{0<t_1<\cdots<t_n<t\}}
p_{t-t_n}(x-x_n)
\prod_{j=1}^{n-1}p_{t_{j+1}-t_j}(x_{j+1}-x_j).
\end{align*}
Denote the  symmetrization  of $f_n$ by $\widetilde f_n$, i.e.,
\begin{equation*}
\widetilde f_n(\boldsymbol t,\boldsymbol x;t,x)
=\frac1{n!}f_n
(t_{\tau(1)},x_{\tau(1)},\ldots,t_{\tau(n)},x_{\tau(n)};t,x),
\end{equation*}
where $\boldsymbol t=(t_1,\ldots,t_n)$,
$\boldsymbol x=(x_1,\ldots,x_n)$ and $\tau\in S_n$ is the unique permutation such that
$
t_{\tau(1)}<t_{\tau(2)}<\cdots<t_{\tau(n)}.$ Set $f_0=\widetilde f_0=1$.

We next justify the following chaos expansion for the Malliavin derivative:
\begin{equation}\label{chaos for Du}
	D_{r,z}u(t,x)
	=
	\sum_{n\ge1}nI_{n-1}
	\big(
	\widetilde f_n(\cdot,r,z;t,x)
	\big),
	\qquad 0<r<t,\quad z\in\R^d,
\end{equation}
where
\begin{align}\label{re:fg1}
\widetilde f_n(\cdot,r,z;t,x)
&=\frac1n\sum_{j=1}^n h_j^{(n)}(\cdot,r,z;t,x),
\end{align}
with $h_j^{(n)}(\cdot,r,z;t,x)$ being the symmetrization, in the remaining
$n-1$ variables, of
\begin{align*}
&f_j^{(n)}(s_1,y_1,\ldots,s_{n-1},y_{n-1},r,z;t,x)=f_n(s_1,y_1,\ldots,s_{j-1},y_{j-1},r,z,
s_j,y_j,\ldots,s_{n-1},y_{n-1};t,x).
\end{align*}
More precisely, to justify \eqref{chaos for Du}, we first prove that
$u(t,x)\in\mathbb D^{1,2}$ and that, for every $0<r<t$ and
$z\in\R^d$, the series on the right-hand side of \eqref{chaos for Du}
converges in $L^2(\Omega)$. Its limit is $D_{r,z}u(t,x)$.

For $n\ge1$, denote, on $r<t_1<\cdots<t_n<t$,
\begin{align*}
&g_n(t_1,x_1,\ldots,t_n,x_n,r,z;t,x)=p_{t-t_n}(x-x_n)
\prod_{j=1}^{n-1}p_{t_{j+1}-t_j}(x_{j+1}-x_j)
p_{t_1-r}(x_1-z),
\end{align*}
and extend it by zero outside this ordered time region.  Set
$g_0(\cdot,r,z;t,x)=p_{t-r}(x-z)$.

By the definitions, for $1\le j\le n$, we have
\begin{align}\label{decom:fj}
&f_j^{(n)}(s_1,y_1,\ldots,s_{n-1},y_{n-1},r,z;t,x)\notag\\
&=
f_{j-1}(s_1,y_1,\ldots,s_{j-1},y_{j-1};r,z)\times
g_{n-j}(s_j,y_j,\ldots,s_{n-1},y_{n-1},r,z;t,x),
\end{align}
with the above conventions for $f_0$ and $g_0$.


The following auxiliary estimates are adaptations of
\cite[(3.10) and the proof of (3.17)]{NXZ22};
see also \cite[(21) and (28)]{BY23}.
\begin{lemma}\label{lem:full-f-estimate}
	There is $c\ge1$, independent of $n,N,t,x$, such that
	\begin{equation*}\label{eq:full-f-estimate}
		\|\widetilde f_n(\cdot;t,x)\|_{\cX^{\otimes n}}^2
		\le\frac{c^n}{n!}
		\sum_{\ell=0}^n
		\frac{t^\ell D_N^\ell}{\ell!}(4C_N)^{n-\ell}.
	\end{equation*}
\end{lemma}
\begin{lemma}\label{lem:full-g-estimate}
	There is $c\ge1$, independent of $n,N,r,z,t,x$, such that
	\begin{equation*}\label{eq:full-g-estimate}
		\|g_n(\cdot,r,z;t,x)\|_{\cX^{\otimes n}}^2
		\le c^{n+1}p_{t-r}^2(x-z)
		\sum_{\ell=0}^{n}
		\frac{[(t-r)D_N]^\ell}{\ell!}(4C_N)^{n-\ell}.
	\end{equation*}
\end{lemma}
We now use the above estimates to prove the convergence of the
chaos series for the first Malliavin derivative and to establish its
 moment bound. 
 \begin{proposition}\label{prop:full-first-malliavin-derivative}
For every $t>0$ and $x\in\R^d$, we have
$u(t,x)\in\mathbb D^{1,2}$ and  the series in \eqref{chaos for Du} converges in
$L^2(\Omega)$.
For every $T>0$ and $p\ge2$, there exists a constant
$C_{T,p}>0$ such that
\begin{equation}\label{eq:full-first-derivative-bound}
	\|D_{r,z}u(t,x)\|_p
	\le C_{T,p}p_{t-r}(x-z),
	\qquad 0<r<t\le T,\quad x,z\in\R^d.
\end{equation}
\end{proposition}
\begin{proof}
 We first prove the $L^2(\Omega)$-convergence of the series in
\eqref{chaos for Du}. By the orthogonality of Wiener chaoses and \eqref{re:fg1}, 
	\begin{align}\label{first-order-chaos}
	\sum_{n\geq 1}\E\Big|n I_{n-1}(\tilde f_n(\cdot,r,z;t,x))\big|^2&=\sum_{n\ge1}\E\Big|I_{n-1}\Big(\sum_{j=1}^nh_j^{(n)}
	(\cdot,r,z;t,x)\Big)\Big|^2\notag \\&=\sum_{n\geq 1}(n-1)!\big\|\sum_{j=1}^n h_j^{(n)}(\cdot,r,z;t,x)\big\|_{\mathcal H^{\otimes(n-1)}}^2\\
	&\leq\sum_{n\geq 1}C_T^{n-1}(n-1)!\big\|\sum_{j=1}^n h_j^{(n)}(\cdot,r,z;t,x)\big\|_{\mathcal X^{\otimes(n-1)}}^2,\notag
		\end{align}
	where the last inequality follows from  \eqref{eq:full-tensor-embedding}.
Since $h_j^{(n)}$ and $h_k^{(n)}$
have disjoint time supports for $j\ne k$, we have \begin{align*}(n-1)!\big\|\sum_{j=1}^nh_j^{(n)}(\cdot,r,z;t,x)\big\|_{\cX^{\otimes(n-1)}}^2=(n-1)!\sum_{j=1}^n\big\|h_j^{(n)}(\cdot,r,z;t,x)\big\|_{\cX^{\otimes(n-1)}}^2.
\end{align*}
Moreover, for each fixed $j$, $h_j^{(n)}$ is the symmetrization of
$f_j^{(n)}$ with respect to its first $n-1$ variables. Hence, 
\begin{align*}
&\big\|h_j^{(n)}(\cdot,r,z;t,x)\big\|_{\cX^{\otimes(n-1)}}^2
\\&=
\frac{1}{((n-1)!)^2}
\sum_{\sigma\in S_{n-1}}
\Big\|
f_j^{(n)}
(s_{\sigma(1)},y_{\sigma(1)},\ldots,
s_{\sigma(n-1)},y_{\sigma(n-1)},r,z;t,x)
\Big\|_{\cX^{\otimes(n-1)}}^2\\
&=
\frac{1}{(n-1)!}
\big\|
f_j^{(n)}(\cdot,r,z;t,x)
\big\|_{\cX^{\otimes(n-1)}}^2\\&=\frac{1}{(n-1)!}\| f_{j-1}(\cdot;r,z)\|_{\cX^{\otimes(j-1)}}^2
\|g_{n-j}(\cdot,r,z;t,x)\|_{\cX^{\otimes(n-j)}}^2,
\end{align*}
where the last equality follows from \eqref{decom:fj}. Similarly, we have
\[\| f_{j-1}(\cdot;r,z)\|_{\cX^{\otimes(j-1)}}^2
=
(j-1)!
\| \widetilde f_{j-1}(\cdot;r,z)\|_{\cX^{\otimes(j-1)}}^2.\]
Consequently,
\begin{align}\label{decom:hj-norm}
&(n-1)!
\bigg\|
\sum_{j=1}^n h_j^{(n)}(\cdot,r,z;t,x)
\bigg\|_{\cX^{\otimes(n-1)}}^2\notag\\
&=
\sum_{j=1}^n
(j-1)!
\|\widetilde f_{j-1}(\cdot;r,z)\|_{\cX^{\otimes(j-1)}}^2
\|g_{n-j}(\cdot,r,z;t,x)\|_{\cX^{\otimes(n-j)}}^2.
\end{align}
By Lemma~\ref{lem:full-f-estimate} and Lemma~\ref{lem:full-g-estimate}, for $1\leq j\leq n$,
\begin{align*}
&(j-1)!
\|\widetilde f_{j-1}(\cdot;r,z)\|_{\cX^{\otimes(j-1)}}^2
\|g_{n-j}(\cdot,r,z;t,x)\|_{\cX^{\otimes(n-j)}}^2\\
&\le c^np_{t-r}^2(x-z)
\sum_{\ell_1=0}^{j-1}\sum_{\ell_2=0}^{n-j}
\frac{(rD_N)^{\ell_1}}{\ell_1!}
\frac{((t-r)D_N)^{\ell_2}}{\ell_2!}\times(4C_N)^{n-1-\ell_1-\ell_2}\\
&\le c^np_{t-r}^2(x-z)(4C_N)^{n-1}
\exp\left(\frac{rD_N}{4C_N}\right)
\exp\left(\frac{(t-r)D_N}{4C_N}\right)\\
&\le c^np_{t-r}^2(x-z)(4C_N)^{n-1}
\exp\left(\frac{TD_N}{4C_N}\right).
\end{align*}
Since $n\le2^n$, summation in $j$ yields
\begin{equation*}\label{eq:full-first-X-bound}
(n-1)!\bigg\|\sum_{j=1}^nh_j^{(n)}(\cdot,r,z;t,x)
\bigg\|_{\cX^{\otimes(n-1)}}^2
\le (2c)^n p_{t-r}^2(x-z)(4C_N)^{n-1}
\exp\left(\frac{TD_N}{4C_N}\right).
\end{equation*}
Since $C_N\to0$, for  fixed $T$ we can choose sufficiently large $N$ such that $8cC_TC_N<1$. Then,
\begin{align}\label{estmate-hj}
&\sum_{n\ge1}C_T^{n-1}(n-1)!
\bigg\|\sum_{j=1}^nh_j^{(n)}(\cdot,r,z;t,x)
\bigg\|_{\cX^{\otimes(n-1)}}^2\notag\\&\le
2c\exp\left(\frac{TD_N}{4C_N}\right)p_{t-r}^2(x-z)
\sum_{n\ge1}\left(8cC_TC_N\right)^{n-1}<\infty.
\end{align}
Therefore, by \eqref{first-order-chaos} and \eqref{estmate-hj}, the series in
\eqref{chaos for Du} converges in $L^2(\Omega)$. 
Moreover, by a similar argument, we use
\eqref{eq:full-tensor-embedding} and Lemma~\ref{lem:full-f-estimate} and  obtain
\begin{align*}
\sum_{n\ge1}n\,n!
\|\widetilde f_n(\cdot;t,x)\|_{\cH^{\otimes n}}^2
&\le
\exp\left(\frac{TD_N}{4C_N}\right)
\sum_{n\ge1}n\bigl(4cC_TC_N\bigr)^n
<\infty.
\end{align*}
Hence, it follows from  \cite[Proposition 1.2.2]{Nualart} that $u(t,x)\in\mathbb D^{1,2}$. By the standard chaos representation of the Malliavin derivative,
together with the pointwise $L^2(\Omega)$-convergence established above,
\eqref{chaos for Du} provides a pointwise version of $Du(t,x)$. In particular,
\begin{align*}
\E|D_{r,z}u(t,x)|^2
\le C_Tp_{t-r}^2(x-z).
\end{align*}

For every $p\ge2$, the hypercontractivity of Wiener chaoses and
\eqref{eq:full-first-X-bound} give
\begin{align*}
\bigg\|I_{n-1}\Big(\sum_{j=1}^nh_j^{(n)}
(\cdot,r,z;t,x)\Big)\bigg\|_p
&\le(p-1)^{(n-1)/2}
\bigg[(n-1)!
\Big\|\sum_{j=1}^nh_j^{(n)}(\cdot,r,z;t,x)
\Big\|_{\cH^{\otimes(n-1)}}^2\bigg]^{1/2}\\
&\le
\sqrt{2c}\exp\left(\frac{TD_N}{8C_N}\right)
p_{t-r}(x-z)
\big(8c(p-1)C_TC_N\big)^{(n-1)/2}.
\end{align*}
Since $C_N\to0$, we may increase $N=N(T,p,d)$ so that
$8c(p-1)C_TC_N<1$.  Hence
\begin{align*}
\|D_{r,z}u(t,x)\|_p
&\le\sum_{n\ge1}
\bigg\|I_{n-1}\Big(\sum_{j=1}^nh_j^{(n)}
(\cdot,r,z;t,x)\Big)\bigg\|_p\\
&\le \sqrt{2c}\exp\left(\frac{TD_N}{8C_N}\right)
p_{t-r}(x-z)
\sum_{n\ge1}
\left(8c(p-1)C_TC_N\right)^{(n-1)/2}\\
&\le C_{T,p}p_{t-r}(x-z).
\end{align*}
\end{proof}

\subsection{Second-order Malliavin derivative estimate}
\label{subsec:full-second-malliavin-derivative}
We now turn to the second Malliavin derivative. Fix
$0<\theta,r<t$ with $\theta\neq r$ and $w,z\in\R^d$. We shall prove $u(t,x)\in\mathbb D^{2,2}$ and 
\begin{equation}\label{eq:chaos of D2}
	D^2_{(\theta,w),(r,z)}u(t,x)
	=
	\sum_{n\ge2}
	n(n-1)I_{n-2}
	\big(
	\widetilde f_n(\cdot,\theta,w,r,z;t,x)
	\big),
\end{equation}
where the series on the right-hand side converges in $L^2(\Omega)$.

For every $n\geq2$, the kernel
$\widetilde f_n(\cdot,\theta,w,r,z;t,x)$
is obtained by fixing two arguments of the symmetric kernel
$\widetilde f_n(\cdot;t,x)$. More precisely,
\begin{align*}
	&\widetilde f_n
	(t_1,x_1,\ldots,t_{n-2},x_{n-2},\theta,w,r,z;t,x)=
	\frac{1}{n(n-1)}
	\sum_{\substack{1\leq i,j\leq n\\i\neq j}}
	h_{i,j}^{(n)}
	(t_1,x_1,\ldots,t_{n-2},x_{n-2},\theta,w,r,z;t,x),
\end{align*}
where $h_{i,j}^{(n)}$ is the symmetrization, with respect to the
remaining $n-2$ space-time variables, of the kernel obtained by placing
$(\theta,w)$ in the $i$-th position and $(r,z)$ in the $j$-th position
of $f_n$. Suppose first that $0<\theta<r<t$. In this case, only the terms with
$i<j$ are nonzero, and hence
\begin{align}\label{relationsip f-g}
	\widetilde f_n(\cdot,\theta,w,r,z;t,x)
	=
	\frac{1}{n(n-1)}
	\sum_{1\leq i<j\leq n}
	h_{i,j}^{(n)}(\cdot,\theta,w,r,z;t,x).
\end{align}
For $1\leq i<j\leq n$, the kernel $h_{i,j}^{(n)}$ is the
symmetrization of
\begin{align*}
&f_{i,j}^{(n)}
(t_1,x_1,\ldots,t_{n-2},x_{n-2},\theta,w,r,z;t,x)\\
&\quad:=
f_n(t_1,x_1,\ldots,t_{i-1},x_{i-1},
\theta,w,
t_i,x_i,\ldots,t_{j-2},x_{j-2},
r,z,
t_{j-1},x_{j-1},\ldots,t_{n-2},x_{n-2};t,x).
\end{align*}
By the definition of $f_n$, we have
\begin{align*}
f_{i,j}^{(n)}
=
f_{i-1}(\cdot;\theta,w)\,
g_{j-i-1}(\cdot,\theta,w;r,z)\,
g_{n-j}(\cdot,r,z;t,x).
\end{align*}

Since the kernels $\widetilde f_n$ are symmetric in their space--time
arguments, it is enough to consider the case $0<\theta<r<t$. The case
$0<r<\theta<t$ follows by interchanging $(\theta,w)$ and $(r,z)$.
\begin{proposition}\label{prop:full-second-malliavin-derivative}
	For every $t>0$ and $x\in\R^d$, we have
$u(t,x)\in\mathbb D^{2,2}$, and the series in \eqref{eq:chaos of D2} converges in $L^2(\Omega)$. Moreover, for  every $T>0$ and $p\geq2$,
	 there exists a constant $C_{T,p}>0$ such that for $0<\theta<r<t\leq T$, $x,w,z\in\R^d$
\begin{equation}\label{eq:full-second-derivative-bound}
	\|D^2_{(\theta,w),(r,z)}u(t,x)\|_p
	\le C_{T,p}p_{r-\theta}(z-w)p_{t-r}(x-z).
\end{equation}
\end{proposition}
\begin{proof}
		By the isometry of multiple Wiener integrals, \eqref{relationsip f-g} and \eqref{eq:full-tensor-embedding},
		\begin{align}\label{second-order-chaos}
            &\sum_{n\ge2}\E
			\Big| n(n-1)
			I_{n-2}
			\big(
			\widetilde f_n(\cdot,\theta,w,r,z;t,x)
			\big)
			\Big|^2\notag\\
			&=\sum_{n\ge2}
			(n-2)!
			\Big\|
			\sum_{1\leq i<j\leq n}
			h_{i,j}^{(n)}(\cdot,\theta,w,r,z;t,x)
			\Big\|_{\mathcal H^{\otimes(n-2)}}^2\\
			&\leq \sum_{n\ge2}(n-2)!C_T^{n-2}
			\Big\|
			\sum_{1\leq i<j\leq n}
			h_{i,j}^{(n)}(\cdot,\theta,w,r,z;t,x)
			\Big\|_{\mathcal X^{\otimes(n-2)}}^2
			.\notag
		\end{align}
		For each fixed $1\le i<j\le n$, arguing as in the first-order case for
$h_j^{(n)}$, we have
		\begin{align*}
			&(n-2)!
			\Big\|
			\sum_{1\le i<j\le n}
			h_{i,j}^{(n)}(\cdot,\theta,w,r,z;t,x)
			\Big\|_{\cX^{\otimes(n-2)}}^2
			\\
			&\quad=
			\sum_{1\le i<j\le n}
			(i-1)!
			\|\widetilde f_{i-1}(\cdot;\theta,w)\|_{\cX^{\otimes(i-1)}}^2
			\|g_{j-i-1}(\cdot,\theta,w;r,z)\|_{\cX^{\otimes(j-i-1)}}^2
			\|g_{n-j}(\cdot,r,z;t,x)\|_{\cX^{\otimes(n-j)}}^2.
		\end{align*}
		By Lemma~\ref{lem:full-f-estimate} and Lemma~\ref{lem:full-g-estimate}, {for all $1 \le i < j \le n$}
		\begin{align*}
			&(i-1)!
			\|\widetilde f_{i-1}(\cdot;\theta,w)\|_{\cX^{\otimes(i-1)}}^2
			\|g_{j-i-1}(\cdot,\theta,w;r,z)\|_{\cX^{\otimes(j-i-1)}}^2
			\|g_{n-j}(\cdot,r,z;t,x)\|_{\cX^{\otimes(n-j)}}^2
			\\
			&\quad\le
			c^n p_{r-\theta}^2(z-w)p_{t-r}^2(x-z)
			\\
			&\qquad\times
			\sum_{\ell_1=0}^{i-1}
			\sum_{\ell_2=0}^{j-i-1}
			\sum_{\ell_3=0}^{n-j}
			\frac{(\theta D_N)^{\ell_1}}{\ell_1!}
			\frac{((r-\theta)D_N)^{\ell_2}}{\ell_2!}
			\frac{((t-r)D_N)^{\ell_3}}{\ell_3!}
			(4C_N)^{n-2-\ell_1-\ell_2-\ell_3}
			\\
			&\quad=
			c^n p_{r-\theta}^2(z-w)p_{t-r}^2(x-z)
			(4C_N)^{n-2}
			\\
			&\qquad\times
			\sum_{\ell_1=0}^{i-1}
			\frac{1}{\ell_1!}
			\left(\frac{\theta D_N}{4C_N}\right)^{\ell_1}
			\sum_{\ell_2=0}^{j-i-1}
			\frac{1}{\ell_2!}
			\left(\frac{(r-\theta)D_N}{4C_N}\right)^{\ell_2}
			\sum_{\ell_3=0}^{n-j}
			\frac{1}{\ell_3!}
			\left(\frac{(t-r)D_N}{4C_N}\right)^{\ell_3}
			\\
			&\quad\le
			c^n p_{r-\theta}^2(z-w)p_{t-r}^2(x-z)
			(4C_N)^{n-2}
			\exp\left(\frac{tD_N}{4C_N}\right)
			\\
			&\quad\le
			c^n p_{r-\theta}^2(z-w)p_{t-r}^2(x-z)
			(4C_N)^{n-2}
			\exp\left(\frac{TD_N}{4C_N}\right).
		\end{align*}
		
		Since
		$\binom{n}{2}\le 2^n,$
		we obtain 
\begin{equation}\label{eq:full-second-X-bound}
			\begin{aligned}
            &(n-2)!
			\Big\|
			\sum_{1\le i<j\le n}
			h_{i,j}^{(n)}(\cdot,\theta,w,r,z;t,x)
			\Big\|_{\cH^{\otimes(n-2)}}^2\\
				&\leq C_T^{n-2}(n-2)!
				\Big\|
				\sum_{1\le i<j\le n}
				h_{i,j}^{(n)}(\cdot,\theta,w,r,z;t,x)
				\Big\|_{\cX^{\otimes(n-2)}}^2\\&
				\le
				(2c)^nC_T^{n-2}
				p_{r-\theta}^2(z-w)p_{t-r}^2(x-z)
				(4C_N)^{n-2}
				\exp\left(\frac{TD_N}{4C_N}\right).
			\end{aligned}
		\end{equation}
         By  \eqref{second-order-chaos} and 
		\eqref{eq:full-second-X-bound}, 
        \begin{align*}
        &\sum_{n\ge2}\E
			\Big| n(n-1)
			I_{n-2}
			\big(
			\widetilde f_n(\cdot,\theta,w,r,z;t,x)
			\big)
			\Big|^2\\&\leq \sum_{n\ge 2}(2c)^nC_T^{n-2}
				p_{r-\theta}^2(z-w)p_{t-r}^2(x-z)
				(4C_N)^{n-2}
				\exp\left(\frac{TD_N}{4C_N}\right).
        \end{align*}
        Since $C_N\to0$, we may choose $N=N(T,d)$ sufficiently large so that
$8cC_TC_N<1$. Therefore,
\begin{align*}
	\sum_{n\ge2}
	(2c)^nC_T^{n-2}(4C_N)^{n-2}
	=
	4c^2\sum_{n\ge2}
	(8cC_TC_N)^{n-2}=
	\frac{4c^2}{1-8cC_TC_N}
	<\infty.
\end{align*}
Consequently, the series in \eqref{eq:chaos of D2} converges in
$L^2(\Omega)$. Moreover, by \eqref{eq:full-tensor-embedding} and
Lemma~\ref{lem:full-f-estimate},
\begin{align*}
\sum_{n\ge2}n(n-1)n!\,
\|\widetilde f_n(\cdot;t,x)\|_{\cH^{\otimes n}}^2<\infty.
\end{align*}
Hence, $u(t,x)\in\mathbb D^{2,2}$, and \eqref{eq:chaos of D2} gives a version of its 
second Malliavin derivative. In particular, there exists a constant $C_{T,2}>0$ such that
\begin{align*}
\E\Big|D^2_{(\theta,w),(r,z)}u(t,x)\Big|^2
\le
C_Tp_{r-\theta}^2(z-w)p_{t-r}^2(x-z).
\end{align*}

     For every $p\ge2$, by hypercontractivity, we have
			\begin{align*}
				&\Big\|
				I_{n-2}\Big(
				\sum_{1\le i<j\le n}
				h_{i,j}^{(n)}(\cdot,\theta,w,r,z;t,x)
				\Big)
				\Big\|_p
				\\
				&\quad\le
				(p-1)^{\frac{n-2}{2}}
				\Big[
				(n-2)!
				\Big\|
				\sum_{1\le i<j\le n}
				h_{i,j}^{(n)}(\cdot,\theta,w,r,z;t,x)
				\Big\|_{\cH^{\otimes(n-2)}}^2
				\Big]^{1/2}
				\\
				&\quad\le
				2c\exp\left(\frac{TD_N}{8C_N}\right)
				\left[
				8c(p-1)C_TC_N
				\right]^{\frac{n-2}{2}}
				p_{r-\theta}(z-w)p_{t-r}(x-z).
			\end{align*}
			Since $C_N\to0$, we may also choose $N$ sufficiently large such that
			$8c(p-1)C_TC_N<1.$
			It follows that
			\begin{align*}
				\|D^2_{(\theta,w),(r,z)}u(t,x)\|_p
				&\le
				\sum_{n\ge2}
				\Big\|
				I_{n-2}\Big(
				\sum_{1\le i<j\le n}
				h_{i,j}^{(n)}(\cdot,\theta,w,r,z;t,x)
				\Big)
				\Big\|_p
				\\
				&\le
				C_{T,p}
				p_{r-\theta}(z-w)p_{t-r}(x-z).
			\end{align*}
	\end{proof}
    \section{Proofs of the main results}\label{sec:main-results}
In this section, we prove the four assertions of Theorem~\ref{thm:main} in their stated
order.

\subsection{Spatial stationarity and ergodicity}
\label{subsec:spatial-ergodicity}
We now prove assertion~\textup{(i)}. 

\begin{proof}[Proof of Theorem~\ref{thm:main}\textup{(i)}]
\emph{Strict stationarity.}
It suffices to show that, for every $m\geq1$ and
$x_1,\ldots,x_m,h\in\mathbb R^d$,
\begin{align*}
	\big(u(t,x_1),\ldots,u(t,x_m)\big)
	\overset{d}{=}
	\big(u(t,x_1+h),\ldots,u(t,x_m+h)\big).
\end{align*}
For $\varphi\in\mathcal H$, define
\begin{align*}
W^h(\varphi):=W(\varphi_h),
\qquad
\varphi_h(s,y):=\varphi(s,y-h).
\end{align*}
By the definition of $\cH$-norm in \eqref{eq-H norm},
\begin{align*}
\langle \varphi_h,\psi_h\rangle_{\mathcal H}
&=
\int_{\mathbb R_+^2}\int_{\mathbb R^{2d}}
\varphi(s,x-h)\psi(r,y-h)
p_{|s-r|}(x-y)dxdydsdr
=
\langle\varphi,\psi\rangle_{\mathcal H}.
\end{align*}
Therefore, $W^h$ is an isonormal Gaussian process over
$\mathcal H$ and
$W^h\overset{d}{=}W.$
It follows from the chaos expansion \eqref{eq:full-solution-chaos} that
\begin{align*}
u(t,x_j+h)
={}&1+\sum_{n\geq1}
\int_{\{0<s_1<\cdots<s_n<t\}}
\int_{\mathbb R^{dn}}
p_{t-s_n}(x_j+h-y_n)\\
&\quad\times
\prod_{i=1}^{n-1}
p_{s_{i+1}-s_i}(y_{i+1}-y_i)
W(ds_1,dy_1)\cdots W(ds_n,dy_n).
\end{align*}
Making the change of variables
$y_i'=y_i-h$, $1\leq i\leq n$, we obtain
\begin{align*}
u(t,x_j+h;W)=u(t,x_j;W^h).
\end{align*}
Consequently,
\begin{align*}
\big(u(t,x_1+h;W),\ldots,u(t,x_m+h;W)\big)
&=
\big(u(t,x_1;W^h),\ldots,u(t,x_m;W^h)\big)\\
&\overset{d}{=}
\big(u(t,x_1;W),\ldots,u(t,x_m;W)\big),
\end{align*}
which proves the strict stationarity.
\medskip

\noindent
\emph{Ergodicity.}
Let
$	U_R
	=
	\int_{[-R,R]^d}
	\cos\Big(\sum_{j=1}^k b_ju(t,x+\xi_j)\Big)dx,$
where $k\geq1$,
$b_1,\ldots,b_k\in\mathbb R$, and
$\xi_1,\ldots,\xi_k\in\mathbb R^d$ are arbitrary.
By the chain rule of the Malliavin derivative,
\begin{align*}
D_{r,y}U_R
=
-\int_{[-R,R]^d}
\sin\bigg(\sum_{j=1}^k b_ju(t,x+\xi_j)\bigg)
\sum_{j=1}^k b_jD_{r,y}u(t,x+\xi_j)\,dx.
\end{align*}
By the Gaussian Poincar\'e inequality,
\begin{align*}
\operatorname{Var}(U_R)
\leq
\mathbb E\|DU_R\|_{\mathcal H}^2.
\end{align*}
Using the definition of the norm in $\mathcal H$, the fact that
$D_{r,y}u(t,x)=0$ for $r>t$, and the Cauchy--Schwarz inequality, we
obtain
\begin{align*}
\mathbb E\|DU_R\|_{\mathcal H}^2
\leq{}&
\int_0^t\int_0^t\int_{\mathbb R^{2d}}
\int_{[-R,R]^d}\int_{[-R,R]^d}
\Big\|
\sum_{j=1}^k b_jD_{s,y}u(t,x+\xi_j)
\Big\|_2\\
&\quad\times
\Big\|
\sum_{l=1}^k b_lD_{r,z}u(t,x'+\xi_l)
\Big\|_2
p_{|s-r|}(y-z)
\,dx\,dx'\,dy\,dz\,ds\,dr.
\end{align*}
By the estimate \eqref{eq:full-first-derivative-bound} for the Malliavin derivative,
\begin{align*}
\Big\|
\sum_{j=1}^k b_jD_{s,y}u(t,x+\xi_j)
\Big\|_2
\leq
C\sum_{j=1}^k|b_j|
p_{t-s}(x+\xi_j-y).
\end{align*}
Therefore,
\begin{align*}
\mathbb E\|DU_R\|_{\mathcal H}^2
\leq{}&
C\sum_{j,l=1}^k|b_jb_l|
\int_0^t\int_0^t\int_{\mathbb R^{2d}}
\int_{[-R,R]^d}\int_{[-R,R]^d}\\
&\quad\times
p_{t-s}(x+\xi_j-y)
p_{t-r}(x'+\xi_l-z)
p_{|s-r|}(y-z)
\,dx\,dx'\,dy\,dz\,ds\,dr.
\end{align*}
Since
$\int_{[-R,R]^d}
p_{t-r}(x'+\xi_l-z)\,dx'
\leq1,$
and
$\int_{\mathbb R^d}p_{|s-r|}(y-z)\,dz=1,$ $
\int_{\mathbb R^d}p_{t-s}(x+\xi_j-y)dy=1,$
we have
\begin{align*}
\mathbb E\|DU_R\|_{\mathcal H}^2
\leq
C\int_0^t\int_0^t\int_{[-R,R]^d}dx\,ds\,dr
\leq
C_tR^d.
\end{align*}
It follows that
\begin{align*}
0
\leq
\operatorname{Var}\left(
\frac{1}{R^d}U_R
\right)
\leq
\frac{C_t}{R^d}
\longrightarrow0
\qquad\text{as }R\to\infty.
\end{align*}

The same argument applies when the cosine function is replaced by the
sine function. Therefore, by \cite[Lemma~4.2]{24bz}, the strictly
stationary random field
$\{u(t,x):x\in\mathbb R^d\}$ is ergodic.
\end{proof}

\subsection{Covariance asymptotics}
\label{subsec:covariance-asymptotics}
\begin{proof}[Proof of Theorem~\ref{thm:main}\textup{(ii)}] It follows from the proof of \cite[Theorem 1.3(i)]{BalanYuan22Wave} that
it is enough to prove 
\begin{align}\label{L1 of rho}
	\int_{\R^d}\rho_{t,s}(z)\,dz<\infty.
\end{align}
Indeed, 
\begin{align*}
\frac{\E\big[F_R(t)F_R(s)\big]}{R^d}
&=
\frac{1}{R^d}
\int_{B_R}\int_{B_R}\rho_{t,s}(x-y)dxdy=
\int_{\R^d}
\frac{|B_R\cap(B_R-z)|}{R^d}\rho_{t,s}(z)\,dz.
\end{align*}
For every fixed $z\in\R^d$,
$\frac{|B_R\cap(B_R-z)|}{R^d}
=
\big|B_1\cap(B_1-R^{-1}z)\big|
\longrightarrow
|B_1|=\omega_d$
as $R\to\infty$, while
\begin{align*}
0\le
\frac{|B_R\cap(B_R-z)|}{R^d}
\le \omega_d.
\end{align*}
If we have $\rho_{t,s}\in L^1(\R^d)$, the dominated convergence theorem yields
\begin{align*}
\lim_{R\to\infty}
\frac{1}{R^d}\E\big[F_R(t)F_R(s)\big]
&=
\omega_d\int_{\R^d}\rho_{t,s}(z)dz=K(t,s).
\end{align*}
Next, we prove \eqref{L1 of rho}.
 Fix $d<4$, $T>0$, and
$s,t\in(0,T]$. By the orthogonality of Wiener chaos,
\begin{align*}
\rho_{t,s}(z)
&=
\sum_{n\geq1}
n!\left\langle
\widetilde f_n(\cdot;t,z),
\widetilde f_n(\cdot;s,0)
\right\rangle_{\mathcal H^{\otimes n}}
=
\sum_{n\geq1}\frac{\alpha_n(z;t,s)}{n!},
\end{align*}
where
\begin{align*}
\alpha_n(z;t,s)
:=
(n!)^2
\left\langle
\widetilde f_n(\cdot;t,z),
\widetilde f_n(\cdot;s,0)
\right\rangle_{\mathcal H^{\otimes n}}
\geq0.
\end{align*}
By the definition of the inner product in $\mathcal H$ and
Cauchy--Schwarz
inequality,
\begin{align}\label{eq:es-rho}
0&\leq\rho_{t,s}(z)
\leq
\sum_{n\geq1}\frac{n\alpha_n(z;t,s)}{n!}\notag\\
&=
\int_0^t\int_0^s\int_{\mathbb R^{2d}}
\sum_{n\geq1}n^2(n-1)!
\left\langle
\widetilde f_n(\cdot,r,y;t,z),
\widetilde f_n(\cdot,\theta,w;s,0)
\right\rangle_{\mathcal H^{\otimes(n-1)}}
p_{|r-\theta|}(y-w)dydwd\theta dr\notag\\
&\leq\int_0^t\int_0^s\int_{\mathbb R^{2d}}
\bigg(
\sum_{n\geq1}n^2(n-1)!
\Big\|
\widetilde f_n(\cdot,r,y;t,z)
\Big\|_{\mathcal H^{\otimes(n-1)}}^2
\bigg)^{1/2}\\
&\qquad\quad\times
\bigg(
\sum_{n\geq1}n^2(n-1)!
\Big\|
\widetilde f_n(\cdot,\theta,w;s,0)
\Big\|_{\mathcal H^{\otimes(n-1)}}^2
\bigg)^{1/2}p_{|r-\theta|}(y-w)dydwd\theta dr\notag\\&=\int_0^t\int_0^s\int_{\mathbb R^{2d}}\Big(\E\Big|D_{r,y}u(t,z)\Big|^2\Big)^{\frac 12}\Big(\E\Big|D_{\theta,w}u(s,0)\Big|^2\Big)^{\frac 12}p_{|r-\theta|}(y-w)dydwd\theta dr,\notag
\end{align}
where the last equality follows from \eqref{chaos for Du}.
Then, by Proposition~\ref{prop:full-first-malliavin-derivative}, 
 there exists a constant $C_T>0$ such that
\begin{align*}
0\leq\rho_{t,s}(z)\leq
C_T\int_0^t\int_0^s\int_{\mathbb R^{2d}}
p_{t-r}(z-y)p_{s-\theta}(w)
p_{|r-\theta|}(y-w)dydwd\theta dr.
\end{align*}
Integrating with respect
to $z$ and using the fact that the heat kernel has total mass one,
we obtain
\begin{align*}
\int_{\mathbb R^d}\rho_{t,s}(z)dz
&\leq
C_T\int_0^t\int_0^s\int_{\mathbb R^{3d}}
p_{t-r}(z-y)p_{s-\theta}(w)
p_{|r-\theta|}(y-w)
dzdydwd\theta dr\\
&=
C_Tts<\infty.
\end{align*}
This proves assertion~\textup{(ii)}.
\end{proof}

\subsection{Quantitative central limit theorem}\label{subsec:quantitative-clt}
For a measurable function $f:\R_+\times\R^d\to\R$, define
\begin{align*}
\|f\|_{|\mathcal H|}^2
:=
\int_{\R_+^2}\int_{\R^{2d}}
|f(r,z)|\,|f(s,y)|\,p_{|r-s|}(z-y)
\,dz\,dy\,dr\,ds.
\end{align*}
We denote by $|\mathcal H|$ the space of measurable functions $f$ such that
$\|f\|_{|\mathcal H|}<\infty$. The space $|\mathcal H^{\otimes 2}|$ is defined
analogously, with the product kernel
$p_{|\theta-\theta'|}(w-w')\,p_{|r-r'|}(z-z').$

We first record the second-order Gaussian Poincaré inequality used in the proof later.
The following proposition provides, in our framework, an analogue of
\cite[Proposition 1.8]{bnz22}.

\begin{proposition}\label{second-order-Poincare}
Let $F\in\mathbb D^{2,4}$ be a centered random variable with
$\operatorname{Var}(F)=\sigma^2\in(0,\infty)$. Assume that, almost
surely, $DF\in|\mathcal H|$ and
$D^2F\in|\mathcal H^{\otimes2}|$. Then
\begin{align}
d_{\mathrm{TV}}(F,N)
\leq
\frac{4}{\sigma^2}\sqrt{\mathcal A(F)},
\label{e:second-order-Poincare}
\end{align}
where $N\sim \mathcal N(0,\sigma^2)$ and
\begin{align}\label{def-cA}
\mathcal A(F)
:={}&
\int_{\mathbb R_+^6}\int_{\mathbb R^{6d}}
\left\|
D^2_{(\theta,w),(r,z)}F
\right\|_4
\left\|
D^2_{(\theta',w'),(s,y)}F
\right\|_4
\left\|
D_{r',z'}F
\right\|_4
\left\|
D_{s',y'}F
\right\|_4
\\
&\quad\times p_{|\theta-\theta'|}(w-w')
p_{|r-r'|}(z-z')
p_{|s-s'|}(y-y')
d\theta d\theta' dr dr' ds ds'
dw dw' dz dz'dy dy'.\notag
\end{align}
\end{proposition}

\begin{proof}
By the Malliavin--Stein inequality (see, e.g.,
\cite{NP12}), 
\begin{align}
d_{\mathrm{TV}}(F,N)
\leq
\frac{2}{\sigma^2}
\mathbb E
\left|
\sigma^2-
\left\langle
DF,-DL^{-1}F
\right\rangle_{\mathcal H}
\right|.
\label{e:Malliavin-Stein}
\end{align}
By the duality relation between $D$ and $\delta$, together with \eqref{e:F-delta-DL}, we have
\begin{align*}
\mathbb E
\left[
\left\langle
DF,-DL^{-1}F
\right\rangle_{\mathcal H}
\right]
=
\mathbb E
\left[
F\delta\left(-DL^{-1}F\right)
\right]=
\mathbb E[F^2]
=
\sigma^2.
\end{align*}
Hence, by the Cauchy--Schwarz inequality,
\begin{align}
d_{\mathrm{TV}}(F,N)
\leq
\frac{2}{\sigma^2}
\sqrt{
\operatorname{Var}
\left(
\left\langle
DF,-DL^{-1}F
\right\rangle_{\mathcal H}
\right)
}.
\label{e:TV-variance}
\end{align}
By the Gaussian Poincar\'e inequality (see, e.g.,
\cite{Nualart,NP12}), 
\begin{align}\label{e:variance-decomposition}
&\operatorname{Var}
\left(
\left\langle
DF,-DL^{-1}F
\right\rangle_{\mathcal H}
\right)\leq
\mathbb E
\left[
\left\|
D\left\langle
DF,-DL^{-1}F
\right\rangle_{\mathcal H}
\right\|_{\mathcal H}^2
\right]\notag\\&\quad\leq
2\mathbb E
\left[
\left\|
\left\langle
D^2F,-DL^{-1}F
\right\rangle_{\mathcal H}
\right\|_{\mathcal H}^2
\right]+
2\mathbb E
\left[
\left\|
\left\langle
DF,-D^2L^{-1}F
\right\rangle_{\mathcal H}
\right\|_{\mathcal H}^2
\right].
\end{align}
Using the representation
$-DL^{-1}F
=
\int_0^\infty e^{-u}P_uDFdu,$ with $\{P_u,u\ge 0\}$ is the Ornstein--Uhlenbeck semigroup 
and Jensen's inequality, we have
\begin{align*}
&\mathbb E
\left[
\left\|
\left\langle
D^2F,-DL^{-1}F
\right\rangle_{\mathcal H}
\right\|_{\mathcal H}^2
\right]\leq
\int_0^\infty e^{-u}
\mathbb E
\left[
\left\|
\left\langle
D^2F,P_uDF
\right\rangle_{\mathcal H}
\right\|_{\mathcal H}^2
\right]du.
\end{align*}
Expanding the inner products in $\mathcal H$ yields that
\begin{align*}
&\mathbb E
\left[
\left\|
\left\langle
D^2F,P_uDF
\right\rangle_{\mathcal H}
\right\|_{\mathcal H}^2
\right]
=
\int_{\mathbb R_+^6}\int_{\mathbb R^{6d}}
p_{|\theta-\theta'|}(w-w')
p_{|r-r'|}(z-z')
p_{|s-s'|}(y-y')\\
&\quad\times
\mathbb E\Big[
D^2_{(\theta,w),(r,z)}F
D^2_{(\theta',w'),(s,y)}F
P_u(D_{r',z'}F)
P_u(D_{s',y'}F)
\Big]
d\theta d\theta'drdr'dsds'
dwdw'dzdz'dydy'.
\end{align*}
By H\"older's inequality and the contraction property of $P_u$ in
$L^4(\Omega)$,
\begin{align*}
&\left|
\mathbb E\Big[
D^2_{(\theta,w),(r,z)}F
D^2_{(\theta',w'),(s,y)}F
P_u(D_{r',z'}F)
P_u(D_{s',y'}F)
\Big]
\right|\\
&\leq
\left\|
D^2_{(\theta,w),(r,z)}F
\right\|_4
\left\|
D^2_{(\theta',w'),(s,y)}F
\right\|_4
\left\|
D_{r',z'}F
\right\|_4
\left\|
D_{s',y'}F
\right\|_4.
\end{align*}
It follows that
\begin{align}
\mathbb E
\left[
\left\|
\left\langle
D^2F,-DL^{-1}F
\right\rangle_{\mathcal H}
\right\|_{\mathcal H}^2
\right]
\leq
\mathcal A(F).
\label{e:first-term-A}
\end{align}

Similarly, using
$-D^2L^{-1}F
=
\int_0^\infty e^{-2u}P_uD^2Fdu,$
and repeating the same argument, we obtain
\begin{align}
\mathbb E
\left[
\left\|
\left\langle
DF,-D^2L^{-1}F
\right\rangle_{\mathcal H}
\right\|_{\mathcal H}^2
\right]
\leq
\mathcal A(F).
\label{e:second-term-A}
\end{align}
Combining \eqref{e:variance-decomposition},
\eqref{e:first-term-A}, and \eqref{e:second-term-A}, we get
\begin{align*}
\operatorname{Var}
\left(
\left\langle
DF,-DL^{-1}F
\right\rangle_{\mathcal H}
\right)
\leq
4\mathcal A(F).
\end{align*}
Substituting this estimate into \eqref{e:TV-variance} yields
\begin{align*}
d_{\mathrm{TV}}(F,N)
\leq
\frac{4}{\sigma^2}\sqrt{\mathcal A(F)},
\end{align*}
which proves the result.
\end{proof}

\begin{proof}[Proof of Theorem~\ref{thm:main}\textup{(iii)}] {Fix $t>0$ and $R>0$. By
Propositions~\ref{prop:full-first-malliavin-derivative}
and~\ref{prop:full-second-malliavin-derivative},
together with Minkowski's inequality, we have
$F_R(t)\in\mathbb D^{2,4}$.
Moreover, since $d<4$, we have 
\begin{align*}
&\mathbb E\left[
\big\|\,|DF_R(t)|\,\big\|_{\mathcal H}^2
+
\big\|\,|D^2F_R(t)|\,\big\|_{\mathcal H^{\otimes2}}^2
\right]\leq C_{t,d}|B_R|<\infty.
\end{align*}
 Hence $DF_R(t)\in|\mathcal H|$ and
$D^2F_R(t)\in|\mathcal H^{\otimes2}|$ almost surely.
Finally, $\mathbb E[F_R(t)]=0$ because $\mathbb E[u(t,x)]=1$,
and the first chaos contribution gives $\sigma_R^2(t)>0$.
Thus Proposition~\ref{second-order-Poincare} applies to $F_R(t)$.}

 Applying  Proposition  \ref{second-order-Poincare} with $F=F_R(t)$ yields that
\begin{align*}
d_{TV}\left(\frac{F_R(t)}{\sigma_R(t)},Z\right)=d_{TV}\left(F_R(t),N\right)\leq\frac{4}{\sigma^2_R(t)}\sqrt{\mathcal A(F_R(t))},
\end{align*}
where $ Z\sim \mathcal N(0,1)$ and $N\sim \mathcal N(0,\sigma_R^2(t))$. 

{By assertion~\textup{(ii)},
\[
\frac{\sigma_R^2(t)}{R^d}\longrightarrow K(t,t)
\qquad\text{as }R\to\infty.
\]
Moreover, retaining only the first chaos contribution gives
$K(t,t)\geq \omega_dt^2>0.$
Therefore, for all sufficiently large $R$,
$\sigma_R^2(t)\geq \frac12K(t,t)R^d.$
}
Hence,
we only need to prove
\begin{align}\label{cA}
{\mathcal A(F_R(t))}\leq C_tR^d.
\end{align}
By Minkowski inequality and Proposition \ref{prop:full-second-malliavin-derivative}, we have 
\begin{align*}
&\|D^2_{(\theta,w),(r,z)} F_{R}(t)\|_4=\Big\|\int_{B_R}D^2_{(\theta,w),(r,z)} u(t,x)dx\Big\|_4\leq \int_{B_R}\|D^2_{(\theta,w),(r,z)} u(t,x)\|_4dx\\&\leq C\int_{B_R}\Big(p_{r-\theta}(z-w)p_{t-r}(x-z)\mathbf 1_{\{0<\theta<r<t\}}+p_{\theta-r}(w-z)p_{t-\theta}(x-w)\mathbf 1_{\{0<r<\theta<t\}}\Big)dx.
\end{align*}Similarly, by Proposition \ref{prop:full-first-malliavin-derivative}, we have for $0<r'<t$
\begin{align*}
\|D_{r',z'}F_R(t)\|_4\leq C\int_{B_R}p_{t-r'}(x-z')dx.
\end{align*}It follows that
\begin{align}\label{eq:cA}
\mathcal A(F_R(t))\leq C_t(\mathcal A_1+\mathcal A_2+\mathcal A_3+\mathcal A_4)
\end{align}
with
\begin{align}\label{a1-a4}
\mathcal A_1=&\int_{\R^{6d}}\int_0^t\int_0^t\int_{\{0<\theta<r<t\}}\int_{\{0<\theta'<s<t\}}\int_{B_R^4} p_{r-\theta}(z-w)p_{t-r}(x_1-z)\notag\\&\times p_{s-\theta'}(y-w')p_{t-s}(x_2-y)p_{t-r'}(x_3-z')p_{t-s'}(x_4-y')p_{|\theta-\theta'|}(w-w')\notag\\&\times p_{|r-r'|}(z-z')p_{|s-s'|}(y-y') dx_1  dx_2  dx_3  dx_4  d\theta' ds d\theta dr ds'dr' dw dw' dy dy' dz dz',\notag\\
\mathcal A_2=&\int_{\R^{6d}}\int_0^t\int_0^t\int_{\{0<\theta<r<t\}}\int_{\{0<s<\theta'<t\}}\int_{B_R^4} p_{r-\theta}(z-w)p_{t-r}(x_1-z)\notag\\&\times  p_{\theta'-s}(w'-y)p_{t-\theta'}(x_2-w')p_{t-r'}(x_3-z')p_{t-s'}(x_4-y')p_{|\theta-\theta'|}(w-w')\notag\\&\times p_{|r-r'|}(z-z')p_{|s-s'|}(y-y') dx_1  dx_2  dx_3  dx_4  d\theta' ds d\theta dr ds'dr' dw dw' dy dy' dz dz',\\
\mathcal A_3=&\int_{\R^{6d}}\int_0^t\int_0^t\int_{\{0<r<\theta<t\}}\int_{\{0<\theta'<s<t\}}\int_{B_R^4} p_{\theta-r}(w-z)p_{t-\theta}(x_1-w)\notag\\&\times p_{s-\theta'}(y-w')p_{t-s}(x_2-y)p_{t-r'}(x_3-z')p_{t-s'}(x_4-y')p_{|\theta-\theta'|}(w-w')\notag\\&\times p_{|r-r'|}(z-z')p_{|s-s'|}(y-y') dx_1  dx_2  dx_3  dx_4  d\theta' ds d\theta dr ds'dr' dw dw' dy dy' dz dz',\notag\\
\mathcal A_4=&\int_{\R^{6d}}\int_0^t\int_0^t\int_{\{0<r<\theta<t\}}\int_{\{0<s<\theta'<t\}}\int_{B_R^4} p_{\theta-r}(w-z)p_{t-\theta}(x_1-w)\notag\\&\times p_{\theta'-s}(w'-y)p_{t-\theta'}(x_2-w')p_{t-r'}(x_3-z')p_{t-s'}(x_4-y')p_{|\theta-\theta'|}(w-w')\notag\\&\times p_{|r-r'|}(z-z')p_{|s-s'|}(y-y') dx_1  dx_2  dx_3  dx_4  d\theta' ds d\theta dr ds'dr' dw dw' dy dy' dz dz'.\notag
\end{align}
For the term $\mathcal A_1$, we integrate the space variables in the order $x_4,y',x_3,z',x_2,y,w',w,z$, then
\begin{align*}
\mathcal A_1&\leq \int_0^t\int_0^t\int_{\{0<\theta<r<t\}}\int_{\{0<\theta'<s<t\}}\int_{B_R}dx_1 d\theta' ds d\theta dr ds'dr'\leq t^6\omega_d R^d. 
\end{align*}
Using a similar argument for $\mathcal A_2,\mathcal A_3,\mathcal A_4$, we have by \eqref{eq:cA} that $\mathcal A(F_R(t))\leq C_tR^d$. The proof is complete.
\end{proof}
\subsection{Functional central limit theorem}
\label{subsec:functional-clt}

\begin{proof}[Proof of Theorem~\ref{thm:main}\textup{(iv)}]
Fix $T>0$ and set
\begin{align*}
Q_R(t):=R^{-d/2}F_R(t),\qquad R>0,\quad t\geq0.
\end{align*}
To prove the convergence in $C[0,T]$, it suffices to
establish the
convergence of the finite-dimensional distributions of $Q_R$ to those
of $Q$ and the tightness of $\{Q_R:R>0\}$ in $C[0,T]$.

\medskip
\noindent
{\bf Step 1. Finite-dimensional distributions.}
Fix $m\geq1$, $0\leq t_1<\cdots<t_m\leq T$, and
$a_1,\ldots,a_m\in\mathbb R$. Set
\begin{align*}
S_R:=\sum_{\ell=1}^m a_\ell F_R(t_\ell).
\end{align*}
By assertion~\textup{(ii)}, as $R\to \infty$,
\begin{align}
R^{-d}\operatorname{Var}(S_R)
&=
\sum_{k,\ell=1}^m a_ka_\ell
R^{-d}\mathbb E[F_R(t_k)F_R(t_\ell)]
\longrightarrow
\sum_{k,\ell=1}^m a_ka_\ell K(t_k,t_\ell)
=:v_a^2\geq0.
\label{e:linear-variance}
\end{align}
By the triangle inequality, Minkowski's inequality, and
Proposition~\ref{prop:full-first-malliavin-derivative} with $p=4$,
we obtain
\begin{align*}
\|D_{r',z'}S_R\|_4
&=
\left\|
\sum_{k=1}^m a_k
\int_{B_R}D_{r',z'}u(t_k,x)\,dx
\right\|_4\\
&\leq
\sum_{k=1}^m|a_k|
\left\|
\int_{B_R}D_{r',z'}u(t_k,x)\,dx
\right\|_4\\
&\leq
\sum_{k=1}^m|a_k|
\int_{B_R}\|D_{r',z'}u(t_k,x)\|_4\,dx\\
&\leq
C_T\sum_{k=1}^m|a_k|
\mathbf1_{\{0<r'<t_k\}}
\int_{B_R}p_{t_k-r'}(x-z')\,dx.
\end{align*}
Similarly, we can obtain by Proposition~\ref{prop:full-second-malliavin-derivative},
\begin{align*}
\|D^2_{(\theta,w),(r,z)}S_R\|_4
&\leq
C_T\sum_{k=1}^m|a_k|
\int_{B_R}
\Big[
\mathbf1_{\{0<\theta<r<t_k\}}
p_{r-\theta}(z-w)p_{t_k-r}(x-z)\\
&\hspace{5em}+
\mathbf1_{\{0<r<\theta<t_k\}}
p_{\theta-r}(w-z)p_{t_k-\theta}(x-w)
\Big]dx.
\end{align*}

Recall that $\mathcal A(F)$ is defined in
Proposition~\ref{second-order-Poincare}. Applying the above estimates
to the four Malliavin derivative factors in $\mathcal A(S_R)$ and
expanding the sums, we obtain
\begin{align*}
\mathcal A(S_R)
\leq
C_T^4
\sum_{i,j,k,\ell=1}^m
|a_i a_j a_k a_\ell|
\sum_{q=1}^4\mathcal A_q^{i,j,k,\ell},
\end{align*}
where $\mathcal A_q^{i,j,k,\ell}$, $1\leq q\leq4$, are the analogues of
$\mathcal A_q$ introduced in \eqref{a1-a4}, with
the four terminal times replaced by $t_i,t_j,t_k,t_\ell$, respectively.
For example,
\begin{align*}
\mathcal A_1^{i,j,k,\ell}
={}&
\int_{\R^{6d}}\int_0^{t_k}\int_0^{t_\ell}
\int_{\{0<\theta<r<t_i\}}
\int_{\{0<\theta'<s<t_j\}}
\int_{B_R^4}
p_{r-\theta}(z-w)p_{t_i-r}(x_1-z)
\\
&\quad\times
p_{s-\theta'}(y-w')p_{t_j-s}(x_2-y)
p_{t_k-r'}(x_3-z')p_{t_\ell-s'}(x_4-y')
\\
&\quad\times
p_{|\theta-\theta'|}(w-w')
p_{|r-r'|}(z-z')
p_{|s-s'|}(y-y')
dx_1dx_2dx_3dx_4
\\
&\quad\times
d\theta'dsd\theta drds'dr'
dwdw'dydy'dzdz'.
\end{align*}
We keep the domain of $x_1$ equal to $B_R$
and enlarge the domains of $x_2,x_3,x_4$ from $B_R$ to $\R^d$.
Then integrating successively with respect to
$x_2, x_3,x_4,z', y', y, w', w, z$ yields that
\begin{align*}
\mathcal A_1^{i,j,k,\ell}
&\leq
|B_R|
\int_0^{t_k}dr'
\int_0^{t_\ell}ds'
\int_{\{0<\theta<r<t_i\}}d\theta dr
\int_{\{0<\theta'<s<t_j\}}d\theta'ds
\leq T^6|B_R|.
\end{align*}
The terms $\mathcal A_2^{i,j,k,\ell}$,
$\mathcal A_3^{i,j,k,\ell}$ and $\mathcal A_4^{i,j,k,\ell}$ can be treated in exactly the same way. Hence, we have
\begin{align}
\mathcal A(S_R)
&\leq
4C_T^4T^6|B_R|
\sum_{i,j,k,\ell=1}^m|a_ia_ja_ka_\ell|
=
4C_T^4T^6\omega_dR^d
\left(\sum_{k=1}^m|a_k|\right)^4
\leq C_{a,T,d}R^d.
\label{e:linear-A}
\end{align}

If $v_a^2>0$, then \eqref{e:linear-variance} gives, for all
sufficiently large $R$,
\begin{align*}
\operatorname{Var}(S_R)\geq\tfrac12v_a^2R^d.
\end{align*}
Since $S_R$ is a finite linear combination of $F_R(t_k)$, the assumptions of Proposition~\ref{second-order-Poincare} are satisfied.
By Proposition~\ref{second-order-Poincare}
and \eqref{e:linear-A},
\begin{align}\label{eq:total-variation-S}
d_{\mathrm{TV}}\left(
\frac{S_R}{\sqrt{\operatorname{Var}(S_R)}},Z
\right)
\leq
\frac{4\sqrt{\mathcal A(S_R)}}{\operatorname{Var}(S_R)}
\leq C_{a,T,d}R^{-d/2}
\longrightarrow0,
\end{align}
where $Z\sim\mathcal N(0,1)$. Hence, by \eqref{e:linear-variance}, \eqref{eq:total-variation-S} and  Slutsky's theorem, as $R\to \infty$
\begin{align*}
R^{-d/2}S_R
&=
\left(R^{-d}\operatorname{Var}(S_R)\right)^{1/2}
\frac{S_R}{\sqrt{\operatorname{Var}(S_R)}}
\overset{d}{\longrightarrow}v_aZ.
\end{align*}
 If $v_a^2=0$, then
\begin{align*}
R^{-d/2}S_R
\overset{d}{\longrightarrow}0=v_aZ.
\end{align*}
The Cram\'er--Wold theorem therefore yields
\begin{align*}
(Q_R(t_1),\ldots,Q_R(t_m))
&\overset{d}{\longrightarrow}
(Q(t_1),\ldots,Q(t_m)),
\end{align*}
where $Q$ is a centered Gaussian process, and 
$\mathbb E[Q(t_i)Q(t_j)]
=K(t_i,t_j)$ for $1\leq i,j\leq m.$

\medskip

\noindent
{\bf Step 2. Tightness.}
Fix $T>0$ and $p>2$, and choose
$\frac1p<\alpha<\frac12.$

Let $0\leq s<t\leq T$. By the chaos expansion,
\begin{align*}
F_R(t)-F_R(s)
=
\sum_{n\geq1}
I_n\left(
g_{n,R}^{(1)}(\cdot;t,s)
+
g_{n,R}^{(2)}(\cdot;t,s)
\right),
\end{align*}
where
\begin{align*}
&g_{n,R}^{(1)}
(t_1,x_1,\ldots,t_n,x_n;t,s)\\
&\quad=
\mathbf1_{\{0<t_1<\cdots<t_n<s\}}
\int_{B_R}
\left[
p_{t-t_n}(x-x_n)-p_{s-t_n}(x-x_n)
\right]
\prod_{j=1}^{n-1}
p_{t_{j+1}-t_j}(x_{j+1}-x_j)\,dx,
\end{align*}
and
\begin{align*}
&g_{n,R}^{(2)}
(t_1,x_1,\ldots,t_n,x_n;t,s)\\
&\quad=
\mathbf1_{\{0<t_1<\cdots<t_n<t,\ s<t_n\}}
\int_{B_R}
p_{t-t_n}(x-x_n)
\prod_{j=1}^{n-1}
p_{t_{j+1}-t_j}(x_{j+1}-x_j)\,dx.
\end{align*}
We first estimate $g_{n,R}^{(1)}$. Set
\begin{align*}
\eta_0=0,
\qquad
\eta_j=\xi_1+\cdots+\xi_j,
\qquad 1\leq j\leq n.
\end{align*}
Taking the Fourier transform with respect to the spatial variables
$x_1,\ldots,x_n$, we obtain
\begin{align*}
&\mathcal Fg_{n,R}^{(1)}
(\mathbf t,\boldsymbol\xi;t,s)\\
&\quad=
\mathbf1_{\{0<t_1<\cdots<t_n<s\}}
\mathcal F\mathbf1_{B_R}(\eta_n)
\left[
e^{-\frac12(t-t_n)|\eta_n|^2}
-
e^{-\frac12(s-t_n)|\eta_n|^2}
\right]
\prod_{j=1}^{n-1}
e^{-\frac12(t_{j+1}-t_j)|\eta_j|^2}.
\end{align*}
Since
$\xi_j=\eta_j-\eta_{j-1},\ 
d\boldsymbol\xi=d\boldsymbol\eta,$
it follows that
\begin{align}\label{eq:normX of g1}
&\|g_{n,R}^{(1)}(\cdot;t,s)\|_{\mathcal X^{\otimes n}}^2\notag\\
&=
c_d^n
\int_{\{0<t_1<\cdots<t_n<s\}}
\int_{\mathbb R^{dn}}
|\mathcal F\mathbf1_{B_R}(\eta_n)|^2
\left|
e^{-\frac12(t-t_n)|\eta_n|^2}
-
e^{-\frac12(s-t_n)|\eta_n|^2}
\right|^2\notag\\
&\hspace{6em}\times
\prod_{j=1}^{n-1}
e^{-(t_{j+1}-t_j)|\eta_j|^2}
\prod_{j=1}^n
\frac{1}{1+|\eta_j-\eta_{j-1}|^2}
d\boldsymbol\eta d\mathbf t\notag\\&\leq c_d^n
\int_{\{0<t_1<\cdots<t_n<s\}}
\int_{\mathbb R^{nd}}
|\mathcal F\mathbf1_{B_R}(\eta_n)|^2
\left|
e^{-\frac12(t-t_n)|\eta_n|^2}
-
e^{-\frac12(s-t_n)|\eta_n|^2}
\right|^2\\
&\hspace{6em}\times
\prod_{j=1}^{n-1}
\bigg(e^{-(t_{j+1}-t_j)|\eta_j|^2}
\frac{1}{1+|\eta_j-\eta_{j-1}|^2}\Big)
d\boldsymbol\eta d\mathbf t\notag\\&=c_d^n
\int_{\{0<t_1<\cdots<t_n<s\}}
\int_{\mathbb R^{d}}
|\mathcal F\mathbf1_{B_R}(\eta_n)|^2
\left|
e^{-\frac12(t-t_n)|\eta_n|^2}
-
e^{-\frac12(s-t_n)|\eta_n|^2}
\right|^2d\eta_n\notag\\
&\hspace{6em}\times
\int_{\R^{d(n-1)}}\prod_{j=1}^{n-1}
e^{-(t_{j+1}-t_j)|\eta_j|^2}
\frac{1}{1+|\eta_j-\eta_{j-1}|^2}
d\eta_{n-1}\cdots d\eta_1 d\mathbf t,\notag
\end{align}
For $t_n<s$, using $|1-e^{-x}|\le C_\alpha x^\alpha$ and noting that
$2\alpha<1$, we obtain
\begin{align}\label{eq:es-minus}
\left|
e^{-\frac12(t-t_n)|\eta_n|^2}
-
e^{-\frac12(s-t_n)|\eta_n|^2}
\right|^2
&=
e^{-(s-t_n)|\eta_n|^2}
\left|
1-e^{-\frac12(t-s)|\eta_n|^2}
\right|^2\notag\\
&\leq
C_\alpha(t-s)^{2\alpha}
|\eta_n|^{4\alpha}
e^{-(s-t_n)|\eta_n|^2}\\
&\leq
C_\alpha(t-s)^{2\alpha}(s-t_n)^{-2\alpha}.\notag
\end{align}
where in the last inequality we used
\[
|\eta_n|^{4\alpha}e^{-(s-t_n)|\eta_n|^2}
=
(s-t_n)^{-2\alpha}
\big[(s-t_n)|\eta_n|^2\big]^{2\alpha}
e^{-(s-t_n)|\eta_n|^2}
\leq
C_\alpha(s-t_n)^{-2\alpha},
\]
since $\sup_{y\geq0}y^{2\alpha}e^{-y}<\infty$.

 For every
$a>0$, the Cauchy--Schwarz inequality gives 
\begin{align*} 
\sup_{\zeta\in\mathbb R^d}
\int_{\mathbb R^d}
\frac{e^{-a|\eta|^2}}{1+|\eta-\zeta|^2}d\eta
&\leq
\left(
\int_{\mathbb R^d}e^{-2a|\eta|^2}\,d\eta
\right)^{1/2}
\sup_{\zeta\in\mathbb R^d}\left(
\int_{\mathbb R^d}
\frac{d\eta}{(1+|\eta-\zeta|^2)^2}
\right)^{1/2}\\
&=
\left(
\int_{\mathbb R^d}e^{-2a|\eta|^2}\,d\eta
\right)^{1/2}
\left(
\int_{\mathbb R^d}
\frac{d\eta}{(1+|\eta|^2)^2}
\right)^{1/2}\\
&\leq C_da^{-d/4},
\end{align*}
where
$\int_{\mathbb R^d}\frac{d\eta}{(1+|\eta|^2)^2}<\infty$
because $d<4$. This gives
\begin{align}\label{e:functional-spectral-bound}
\int_{\R^{d(n-1)}}\prod_{j=1}^{n-1}
e^{-(t_{j+1}-t_j)|\eta_j|^2}
\frac{1}{1+|\eta_j-\eta_{j-1}|^2}
d\eta_{n-1}\cdots d\eta_1\leq C_d^n
\prod_{j=1}^{n-1}
(t_{j+1}-t_j)^{-d/4}.
\end{align}
 Moreover, Plancherel's identity yields
\begin{align}\label{eq:Fourier of B}
\int_{\mathbb R^d}
|\mathcal F\mathbf1_{B_R}(\eta)|^2\,d\eta
=
C_d|B_R|.
\end{align}
Consequently, by \eqref{eq:normX of g1}-\eqref{eq:Fourier of B} and Lemma~\ref{lem:ineq-gamma}, we have
\begin{align*}
\|g_{n,R}^{(1)}(\cdot;t,s)\|_{\mathcal X^{\otimes n}}^2
&\leq
C_{\alpha,d}^n|B_R|(t-s)^{2\alpha}
\int_{\{0<t_1<\cdots<t_n<s\}}
(s-t_n)^{-2\alpha}
\prod_{j=1}^{n-1}
(t_{j+1}-t_j)^{-d/4}
\,d\mathbf t\\
&=
C_{\alpha,d}^n|B_R|(t-s)^{2\alpha}
s^{1-2\alpha+(n-1)(1-d/4)}
\frac{
\Gamma(1-2\alpha)\Gamma(1-d/4)^{n-1}
}{
\Gamma\big(2-2\alpha+(n-1)(1-d/4)\big)
}.
\end{align*}

Note that
\begin{align*}
n!
\|\widetilde g_{n,R}^{(1)}(\cdot;t,s)\|_{\mathcal X^{\otimes n}}^2
=
\|g_{n,R}^{(1)}(\cdot;t,s)\|_{\mathcal X^{\otimes n}}^2.
\end{align*}
Therefore, by Proposition~\ref{lem:full-schur},
\begin{align*}
n!
\|\widetilde g_{n,R}^{(1)}(\cdot;t,s)\|_{\mathcal H^{\otimes n}}^2\leq n!C_T^n
\|\widetilde g_{n,R}^{(1)}(\cdot;t,s)\|_{\mathcal X^{\otimes n}}^2
=
C_T^n
\|g_{n,R}^{(1)}(\cdot;t,s)\|_{\mathcal X^{\otimes n}}^2.
\end{align*}
Since $|B_R|=\omega_dR^d$ and $s\leq T$, it follows that
\begin{align}
\sqrt{n!}
\|\widetilde g_{n,R}^{(1)}(\cdot;t,s)\|_{\mathcal H^{\otimes n}}
\leq
\frac{
C_{\alpha,T,d}^nR^{d/2}(t-s)^\alpha
}{
\sqrt{
\Gamma\big(2-2\alpha+(n-1)(1-d/4)\big)
}
}.
\label{e:functional-first-kernel}
\end{align}

We next estimate $g_{n,R}^{(2)}$.
\begin{align*}
&\|g_{n,R}^{(2)}(\cdot;t,s)\|_{\mathcal X^{\otimes n}}^2=
c_d^n
\int_{\{0<t_1<\cdots<t_n<t,\ s<t_n\}}
\int_{\mathbb R^{dn}}
|\mathcal F\mathbf1_{B_R}(\eta_n)|^2
e^{-(t-t_n)|\eta_n|^2}\\
&\hspace{8em}\times
\prod_{j=1}^{n-1}
e^{-(t_{j+1}-t_j)|\eta_j|^2}
\prod_{j=1}^n
\frac{1}{1+|\eta_j-\eta_{j-1}|^2}
d\boldsymbol\eta d\mathbf t\\&\leq c_d^n
\int_{\{0<t_1<\cdots<t_n<t,\ s<t_n\}}
\int_{\mathbb R^{d}}
|\mathcal F\mathbf1_{B_R}(\eta_n)|^2d\eta_n
\int_{\R^{d(n-1)}}\prod_{j=1}^{n-1}
\frac{e^{-(t_{j+1}-t_j)|\eta_j|^2}}{1+|\eta_j-\eta_{j-1}|^2}
d\eta_{n-1}\cdots d\eta_1 d\mathbf t.
\end{align*}
 Applying \eqref{e:functional-spectral-bound} and \eqref{eq:Fourier of B} yields that 
\begin{align*}
\|g_{n,R}^{(2)}(\cdot;t,s)\|_{\mathcal X^{\otimes n}}^2
&\leq
C_d^n|B_R|
\int_{\{0<t_1<\cdots<t_n<t,\ s<t_n\}}
\prod_{j=1}^{n-1}
(t_{j+1}-t_j)^{-d/4}
d\mathbf t\\
&=
C_d^n|B_R|
\frac{\Gamma(1-d/4)^{n-1}}
{\Gamma\big(1+(n-1)(1-d/4)\big)}
\int_s^t
t_n^{(n-1)(1-d/4)}\,dt_n\\
&\leq
C_d^n|B_R|(t-s)
T^{(n-1)(1-d/4)}
\frac{\Gamma(1-d/4)^{n-1}}
{\Gamma\big(1+(n-1)(1-d/4)\big)}.
\end{align*}
 Therefore,
\begin{align*}
n!
\|\widetilde g_{n,R}^{(2)}(\cdot;t,s)\|^2_{\mathcal H^{\otimes n}}\leq n!C_T^n
\|\widetilde g_{n,R}^{(2)}(\cdot;t,s)\|_{\mathcal X^{\otimes n}}^2
=C_T^n
\|g_{n,R}^{(2)}(\cdot;t,s)\|_{\mathcal X^{\otimes n}}^2.
\end{align*}
Consequently,
\begin{align}
\sqrt{n!}
\|\widetilde g_{n,R}^{(2)}(\cdot;t,s)\|_{\mathcal H^{\otimes n}}
\leq
\frac{
C_{T,d}^nR^{d/2}(t-s)^{1/2}
}{
\sqrt{
\Gamma\big(1+(n-1)(1-d/4)\big)
}
}.
\label{e:functional-second-kernel}
\end{align}
Since $d<4$, Stirling's formula yields that 
\begin{align*}
\sum_{n\ge1}
\frac{C^n}
{\sqrt{\Gamma(c+(n-1)(1-d/4))}}
<\infty.
\end{align*}

For each $n\geq1$, hypercontractivity, the isometry of multiple
Wiener integrals, and the triangle inequality yield
\begin{align*}
\left\|
I_n\left(
g_{n,R}^{(1)}(\cdot;t,s)
+
g_{n,R}^{(2)}(\cdot;t,s)
\right)
\right\|_p
\leq
(p-1)^{n/2}\sqrt{n!}
\left(
\|\widetilde g_{n,R}^{(1)}(\cdot;t,s)\|_{\mathcal H^{\otimes n}}
+
\|\widetilde g_{n,R}^{(2)}(\cdot;t,s)\|_{\mathcal H^{\otimes n}}
\right).
\end{align*}
Applying this estimate first to finite chaos sums and then using
\eqref{e:functional-first-kernel},
\eqref{e:functional-second-kernel}, and the preceding summability to
pass to the limit, we obtain
\begin{align*}
\|F_R(t)-F_R(s)\|_p
\leq
C_{p,T,\alpha,d}R^{d/2}
\left(
(t-s)^\alpha+(t-s)^{1/2}
\right)\leq
C_{p,T,\alpha,d}R^{d/2}|t-s|^\alpha.
\end{align*}

When $s=0$, the kernel $g_{n,R}^{(1)}$ vanishes, while the preceding
estimate for $g_{n,R}^{(2)}$ remains valid. Hence, $Q_R(0)=0$ and 
\begin{align*}
\sup_{R>0}
\mathbb E\left[
|Q_R(t)-Q_R(s)|^p
\right]
&\leq
C_{p,T,\alpha,d}|t-s|^{\alpha p}.
\end{align*}
Since $\alpha p>1$, Kolmogorov's criterion yields continuous
modifications and the tightness of $\{Q_R:R>0\}$ in $C[0,T]$.

\medskip
\noindent
{\bf Step 3. Conclusion.}
By Steps 1 and 2, every subsequential limit in $C[0,T]$
has the finite-dimensional distributions of $Q$. Consequently,
$Q$ admits a continuous modification and
\begin{align*}
Q_R\overset{d}{\longrightarrow}Q
\qquad\text{in }C[0,T].
\end{align*}
Since $T>0$ is arbitrary,
\begin{align*}
R^{-d/2}F_R\overset{d}{\longrightarrow}Q
\qquad\text{in }C[0,\infty),
\end{align*}
where $C[0,\infty)$ is equipped with the topology of uniform
convergence on compact intervals.
\end{proof}

\begin{lemma}\label{lem:ineq-gamma}
	Suppose $\alpha_i<1$ for  $i=1, \dots, m$ and let $\alpha=\sum_{i=1}^m\alpha_i$.
	Then 
	\[\int_{\{0< r_1<\dots<r_m<r_{m+1}=t\}}~ \prod_{i=1}^{m}( r_{i+1}-r_{i})^{-\alpha_i}d \mathbf r = \frac{\prod_{i=1}^m\Gamma(1-\alpha_i)}{\Gamma(m-\alpha+1)} t^{m-\alpha},\]
	where $\Gamma(x)=\int_0^\infty t^{x-1} e^{-t}d t$ is the Gamma function.
\end{lemma}

{\bf Acknowledge}

Wangjun Yuan is supported by a grant of the Department of Science and Technology of Guangdong Province (No.\ 2024QN11X161), the National Natural Science Foundation of China (No.\ 12501183), and the Guangdong Basic and Applied Basic Research Foundation (No.\ 2026A1515030040).

\bibliographystyle{plain}
\bibliography{Reference-clt}
\end{document}